\documentclass[a4paper,reqno]{amsart}

\usepackage[T1]{fontenc}
\usepackage[utf8x]{inputenc}
\usepackage[english]{babel}
\usepackage{yfonts}
\usepackage{dsfont}
\usepackage{amscd,amssymb,amsmath,amsthm,amsfonts}
\usepackage{mathtools}
\usepackage{accents}
\usepackage{tensor}
\usepackage{graphicx}
\usepackage{tikz}
\usetikzlibrary{calc,matrix,arrows,decorations.pathmorphing}
\usepackage{calc}
\usepackage[cal=boondox,scr=boondoxo]{mathalfa}
\usepackage{enumitem}
\usepackage{marginnote}
\usepackage{booktabs}
\usepackage{scalerel}[2016/12/29]
\usepackage{url}
\usepackage{contour}
\usepackage[normalem]{ulem}

\usepackage{hyperref}
\hypersetup{colorlinks=true,pageanchor=false,
linkcolor=blue,citecolor=red,urlcolor=red}
\usepackage[all]{hypcap}

\allowdisplaybreaks

\newtheorem{theorem}{Theorem}[section]
\newtheorem{corollary}[theorem]{Corollary}
\newtheorem{proposition}[theorem]{Proposition}
\newtheorem{lemma}[theorem]{Lemma}

\theoremstyle{definition}

\theoremstyle{remark}
\newtheorem{remark}[theorem]{Remark}

\newcommand{\N}{\mathbb{N}}
\newcommand{\Z}{\mathbb{Z}}
\newcommand{\Q}{\mathbb{Q}}

\newcommand{\rmd}{\mathrm{d}}

\newcommand{\bbH}{\mathbb{H}}

\newcommand{\bfone}{\mathbf{1}}
\newcommand{\bfa}{\boldsymbol{a}}
\newcommand{\bfb}{\boldsymbol{b}}

\newcommand{\bfk}{\boldsymbol{k}}

\newcommand{\bfell}{\boldsymbol{\ell}}
\newcommand{\bfm}{\boldsymbol{m}}
\newcommand{\bfn}{\boldsymbol{n}}

\newcommand{\bft}{\boldsymbol{t}}
\newcommand{\bfv}{\boldsymbol{v}}

\newcommand{\bfx}{\boldsymbol{x}}
\newcommand{\bfy}{\boldsymbol{y}}

\newcommand{\bfE}{\boldsymbol{E}}
\newcommand{\bfF}{\boldsymbol{F}}

\newcommand{\bfGamma}{\boldsymbol{\Gamma}}

\newcommand{\calB}{\mathcal{B}}
\newcommand{\calC}{\mathcal{C}}

\newcommand{\calE}{\mathcal{E}}
\newcommand{\calF}{\mathcal{F}}

\newcommand{\calH}{\mathcal{H}}

\newcommand{\fraki}{\mathfrak{i}}

\newcommand{\fraku}{\mathfrak{u}}

\renewcommand{\epsilon}{\varepsilon}
\renewcommand{\theta}{\vartheta}
\renewcommand{\phi}{\varphi}
\renewcommand{\Gamma}{\varGamma}
\renewcommand{\Sigma}{\varSigma}

\newcommand{\id}{\mathrm{id}}

\newcommand{\im}{\operatorname{im}}

\DeclareMathOperator{\Exists}{\exists}
\DeclareMathOperator{\Forall}{\forall}

\newcommand{\leqs}{\leqslant}
\newcommand{\geqs}{\geqslant}

\newcommand{\fsl}{\mathfrak{sl}}
\newcommand{\Mod}{\mathrm{Mod}}

\newcommand{\BM}{\mathrm{BM}}

\newcommand{\GL}{\mathrm{GL}}
\newcommand{\PGL}{\mathrm{PGL}}
\newcommand{\SU}{\mathrm{SU}}

\newcommand{\sqbinom}[2]{\left[ \begin{matrix} #1 \\ #2 \end{matrix} \right]}

\makeatletter
\newcommand{\subalign}[1]{
  \vcenter{
    \Let@ \restore@math@cr \default@tag
    \baselineskip\fontdimen10 \scriptfont\tw@
    \advance\baselineskip\fontdimen12 \scriptfont\tw@
    \lineskip\thr@@\fontdimen8 \scriptfont\thr@@
    \lineskiplimit\lineskip
    \ialign{\hfil$\m@th\scriptstyle##$&$\m@th\scriptstyle{}##$\crcr
      #1\crcr
    }
  }
}
\makeatother

\def\clap#1{\hbox to 0pt{\hss#1\hss}}

\def\mathclap{\mathpalette\mathclapinternal}

\def\mathclapinternal#1#2{%
\clap{$\mathsurround=0pt#1{#2}$}}

\newcommand{\pic}[2][0]{\raisebox{-0.5\height + 2.5pt + #1pt}{\includegraphics{#2.pdf}}}

\newcommand\arxiv[2]{\href{https://arXiv.org/abs/#1}{\texttt{arXiv:\allowbreak #1} #2}}
\newcommand\doi[2]{\href{https://doi.org/#1}{#2}}

\contourlength{0.625pt}

\DeclareRobustCommand{\myuline}[1]{
 \ifmmode \text{\uline{$\phantom{#1}$}\llap{\contour{white}{$#1$}}}
 \else \uline{\phantom{#1}}\llap{\contour{white}{#1}} \fi
}

\makeatletter
\def\namedlabel#1#2{\begingroup
    #2%
    \def\@currentlabel{#2}%
    \phantomsection\label{#1}\endgroup
}
\makeatother

\newcommand{\twist}[2]{\tau_{#1_{#2}}}

\begin{document}

\raggedbottom

\title[On Integrality of Non-Semisimple Quantum Representations]{On Integrality of Non-Semisimple Quantum Representations of Mapping Class Groups}

\author[M. De Renzi]{Marco De Renzi} 
\address{IMAG, Université de Montpellier, Place Eugène Bataillon, 34090 Montpellier, France}
\email{marco.de-renzi@umontpellier.fr}

\author[J. Martel]{Jules Martel} 
\address{AGM, Cergy Paris Université, Site de Saint-Martin, 2 avenue Adolphe Chauvin, 95300 Pontoise, France} 
\email{jules.martel-tordjman@cyu.fr}

\begin{abstract}
 For a root of unity $\zeta$ of odd prime order, we restrict coefficients of non-semisimple quantum representations of mapping class groups associated with the small quantum group $\fraku_\zeta \fsl_2$ from $\Q(\zeta)$ to $\Z[\zeta]$. We do this by exhibiting explicit bases of states spaces that span $\Z[\zeta]$-lattices that are invariant under projective actions of mapping class groups.
\end{abstract}

\maketitle
\setcounter{tocdepth}{3}

\section{Introduction}

Quantum representations of mapping class groups are one of the most valuable byproducts of TQFTs (Topological Quantum Field Theories). Those constructed by Reshetikhin and Turaev \cite{RT91}, which were inspired by the work of Witten \cite{W89}, will be referred to here as \textit{WRT representations}, or \textit{semisimple quantum representations}. We will focus in particular on projective representations associated with the quantum group of $\fsl_2$ at a root of unity of odd order $r$. Many remarkable properties of these representations are known. They preserve a (positive-definite) Hermitian pairing \cite{T94, BHMV95}, but their image is infinite \cite{F98, M99}. They are asymptothically faithful \cite{A02, FWW02}, in the sense that the intersection of all their kernels, as $r$ takes all possible values, is trivial. Furthermore, their coefficients can be restricted to rings of cyclotomic integers \cite{GMv02, GM04}. This last property in particular is at the heart of many applications of WRT representations: it can be used to show that every finite group is a quotient of a finite index subgroup of the mapping class group of any closed surface of genus at least $2$ \cite{MR11}, or to construct finite covers of surfaces whose homology is not spanned by lifts of simple closed curves \cite{KS15}.

Despite these striking properties, it is easy to find elements in the kernels of WRT representations, since every Dehn twist is sent to a finite-order operator. This is not the case, however, for the \textit{non-semisimple quantum representations} constructed by Kerler \cite{K94} and Lyubashenko \cite{L94}, based on the work of Hennings \cite{H96}, which will be referred to here as \textit{HKL representations}, or \textit{non-semisimple quantum representations} (once again, we will restrict our attention to the quantum $\fsl_2$ case). Indeed, every Dehn twist is sent to a non-diagonalizable infinite-order operator by these representations \cite{DGGPR20}, which therefore have infinite image. Even more can be said, since no diffeomorphism has been found yet in the kernels of these actions, which have been shown to be faithful in genus $1$ \cite{K96}. On top of this, many of the nice properties of semisimple quantum representations generalize to the non-semisimple setting. Indeed, HKL representations preserve a Hermitian pairing too, although an indefinite one \cite{GLPS21}. Furthermore, asymptotic faithfulness is an immediate consequence of a homological model developed in \cite{DM22}, which embeds HKL representations into a direct sum of twisted homological representations appearing in \cite{BPS21}, see Remark~\ref{R:asymptotic_faithfulness} for an explanation.

\subsection{Main result}

This paper is concerned with integrality, whose generalization to quantum $\fsl_2$ HKL representations was not known in general, up to now. In order to state more precisely our main result, we need to introduce some notation. First of all, let $g \geqs 0$ be an integer, and let $r \geqs 3$ be an odd integer. Let us fix a connected surface $\varSigma_g$ of genus $g$ with a single boundary component, whose \textit{mapping class group} will be denoted $\Mod(\varSigma_g)$. Next, let $\fraku_\zeta^\Q = \fraku_\zeta \fsl_2$ denote the \textit{small quantum group of $\fsl_2$} at the root of unity $\zeta = \smash{e^{\frac{2 \pi \fraki}{r}}}$ over the field $\Q(\zeta)$, whose definition is recalled in Section~\ref{S:small_quantum_sl2}. The corresponding HKL representation, whose definition is recalled in Section~\ref{S:quantum_representations}, is given by a homomorphism
\begin{align*}
 \rho^\Q_{g,r} : \Mod(\varSigma_g) \to \PGL_{\fraku_\zeta^\Q}(U_g).
\end{align*}
where $U_g := (\fraku_\zeta \fsl_2)^{\otimes g}$ carries the structure of a $\fraku_\zeta^\Q$-module through the $g$th tensor power of the adjoint action, and where $\smash{\PGL_{\fraku_\zeta^\Q}}(U_g)$ denotes the group of invertible $\fraku_\zeta^\Q$-module endomorphisms, considered up to non-zero scalars in $\Q(\zeta)$. We show that, if $r$ is prime, then it is possible to restrict coefficients from $\Q(\zeta)$ to $\Z[\zeta]$. We do this by considering an integral version $\fraku_\zeta^\Z$ of $\fraku_\zeta^\Q$ defined over $\Z[\zeta]$, and by providing an explicit \textit{integral} basis of $U_g$ for the action of both $\fraku_\zeta^\Z$ and $\Mod(\varSigma_g)$, see Corollary~\ref{C:main_result}.

\begin{theorem}\label{T:main_result}
 If $r$ is prime, then the basis $\calB^{E1'F}_g$ of $U_g$ defined by Equation~\eqref{E:E1primeF_basis} spans a $\Z[\zeta]$-lattice $U'_g \subset U_g$ that is invariant under the action of both $\fraku_\zeta^\Z$ and $\Mod(\varSigma_g)$, and thus yields a homomorphism
 \begin{align*}
  \rho^\Z_{g,r} : \Mod(\varSigma_g) \to \PGL_{\fraku_\zeta^\Z}(U'_g).
 \end{align*}
\end{theorem}
 
It should be specified that some integrality properties in the non-semisimple setting had already been established, thanks to the work of Chen and Kerler \cite{CK13}. Their results, however, only concern those HKL representations that arise from quantum doubles of Hopf algebras, so $\fraku_\zeta^\Q$ does not fit in their framework.

\subsection{Strategy of the proof}\label{S:strategy}

In order to establish our result, we rely on the homological model developed in \cite{DM22}, which allows us to reduce the question to a simpler but equivalent integrality problem, see Remark~\ref{R:integrality}.
Indeed, the state space $U_g$ can be embedded into a $\Q(\zeta)$-vector space $\calH_g^V$ that can be constructed as a direct sum of twisted homology groups of unordered configuration spaces $\calC_{n,g} = \calC_n(\varSigma_g)$ of the surface $\varSigma_g$. Coefficients for these twisted homology groups are provided by the \textit{Schrödinger representation} $\varphi_g : \bbH_g \to \GL_{\Z[\zeta]}(V_g)$ of the \textit{Heisenberg group} $\bbH_g$, which can be identified with a quotient of the \textit{surface braid group} $\pi_{n,g} = \pi_1(\calC_{n,g})$ for $n \geqs 2$. This embedding of $U_g$ into $\calH_g^V$ also intertwines actions of $\fraku_\zeta^\Z$ and $\Mod(\varSigma_g)$ defined on both sides: the ones on $U_g$ are given by the HKL representation, while the ones on $\calH_g^V$ were introduced in \cite{BPS21} and \cite{DM22}.

Let us try to be a little more precise. There exists an identification
\begin{align*}
 \calH_g^V \cong \bigoplus_{n \geqs 0} \calH_{n,g}^\bbH \otimes_{\Z[\bbH_g]} V_g,
\end{align*}
where $\calH_{n,g}^\bbH$ is the $n$th twisted homology group of the configuration space $\calC_{n,g}$ (relative to a portion of its boundary) with coefficients in the group ring $\Z[\bbH_g]$. Twisted homology classes in $\calH_{n,g}^\bbH$ can be explicitly described in terms of diagrams composed of system of curves in the surface $\varSigma_g$, see Section~\ref{S:Heisenberg_homology}. We use these diagrams to construct an action of $\fraku_\zeta^\Z$ onto $\calH_g^V$ that is defined in purely homological terms. Furthermore, a very nice feature of this model is that it comes with computation rules that allow us to express any arbitrary diagram as a $\Z[\bbH_g]$-linear combination of diagrams belonging to a fixed basis of $\calH_{n,g}^\bbH$. We stress the fact that these diagrams actually represent submanifolds of a regular cover $\hat{\calC}_{n,g} = \hat{\calC}_n(\varSigma_g)$ of $\calC_{n,g}$, and computation rules represent homological relations between these submanifolds, so the model is not purely combinatorial.

The problem for integrality then lies entirely with coefficients in $V_g$. Indeed, the action of $f \in \Mod(\varSigma_g)$ onto $\sigma \otimes v \in \calH_{n,g}^\bbH \otimes_{\Z[\bbH_g]} V_g$ is given by
\begin{align*}
 f \cdot (\sigma \otimes v) = (\hat{\calC}_n(f) \cdot \sigma) \otimes (f \cdot v),
\end{align*}
where $\hat{\calC}_n(f) \cdot \sigma$ is determined by the natural action of a self-diffeomorphism onto the homology of its domain (and can be explicitly described by the system of curves in the surface $\varSigma_g$ obtained by applying $f$ to the diagram representing $\sigma$), and where $f \cdot v$ is determined by a projective representation $\psi_g : \Mod(\varSigma_g) \to \GL_{\Q(\zeta)}(V_g)$. This is the only ingredient in the whole construction to feature coefficients in $\Q(\zeta)$ instead of $\Z(\zeta)$ or $\Z[\bbH_g]$.
With this in mind, we simply set out to find a basis of $V_g$ spanning a $\Z[\zeta]$-lattice $V'_g \subset V_g$ that is invariant under the projective action of the mapping class group $\Mod(\varSigma_g)$.

\subsection{Structure of the paper}

In Section~\ref{S:quantum_representations}, we recall the definition of the non-semisimple quantum representations we refer to here as HKL representations. These can be built out of any factorizable ribbon Hopf algebra $H$ (or more generally out of any modular category in the potentially non-semisimple sense of Lyubashenko, although we will not need the notion here). Their construction yields projective actions of the mapping class group $\Mod(\varSigma_g)$ of the connected surface $\varSigma_g$ of genus $g$ with a single boundary component on $H^{\otimes g}$.

In Section~\ref{S:small_quantum_sl2} we recall the definition of the small quantum group $\smash{\fraku_\zeta^\Q} = \fraku_\zeta \fsl_2$ and of its integral version $\fraku_\zeta^\Z$. This provides the concrete example of factorizable ribbon Hopf algebra we will focus on, and the integrality of the associated HKL representations will be established in this paper.

In Section~\ref{S:Heisenberg_homology}, we recall the definition of the Heisenberg group $\bbH_g$ and of the corresponding Heisenberg homology groups $\calH_{n,g}^\bbH$ of the surface $\varSigma_g$, following \cite{BPS21,DM22}. We recall in particular the natural crossed $\Z[\bbH_g]$-action of $\Mod(\varSigma_g)$ carried by $\calH_{n,g}^\bbH$.

In Section~\ref{S:Schrdngr}, we recall how to obtain from this construction a family of projective representations of $\Mod(\varSigma_g)$, by choosing a homomorphism $\varphi_g : \bbH_g \to \GL_{\Z[\zeta]}(V_g)$ called the Schrödinger representation, which comes with an associated projective representation $\psi_g : \Mod(\varSigma_g) \to \PGL_{\Q(\zeta)}(V_g)$ that is responsible for the appearence of both denominators and projective indeterminacies.

In Section~\ref{S:isomorphism}, we recall the main result of \cite{DM22}, which allows us to embed the HKL representations of Sections~\ref{S:quantum_representations}--\ref{S:small_quantum_sl2} into the direct sum of twisted homological representations of Sections~\ref{S:Heisenberg_homology}--\ref{S:Schrdngr}. Thanks to this model, we are able to reduce the integrality problem for $U_g$ to the one for $V_g$.

In Section~\ref{S:triangular_basis}, we define a triangular basis $\calB_g^t$ for $V_g$, which owes its name to the fact that, with respect to $\calB_g^t$, standard Lickorish generators of $\Mod(\varSigma_g)$ all become either upper-triangular or lower-triangular, see Proposition~\ref{P:triangularity}. The proof of this crucial property relies on explicit computations which are carried out in Appendices~\ref{A:quantum_identities} and \ref{A:computations}.

In Section~\ref{S:integral_rescaling}, we define an integral basis $\calB_g^{v'}$ for $V_g$ obtained from $\calB_g^t$ through a diagonal renormalization process that ensures that all non-diagonal coefficients for the action of generators of $\Mod(\varSigma_g)$ belong to $\Z[\zeta]$, see Proposition~\ref{P:divisibility}. Our approach is inspired by \cite{GMv02}. Once again, explicit proofs are postponed to Appendix~\ref{A:divisibility}.

Finally, our main result is proved in Section~\ref{S:main_theorem}, see Theorem~\ref{T:integral_Schrdngr} and Corollary~\ref{C:main_result}.

\subsection*{Acknowledgments} The authors would like to thank Renaud Detcherry and Gregor Masbaum for helpful discussions and encouragements.

\section{HKL representations of mapping class groups}\label{S:quantum_representations}

A Hopf algebra over a field $\Bbbk$ is a $\Bbbk$-vector space $H$ equipped with a family of $\Bbbk$-linear maps composed of a unit $\eta : \Bbbk \to H$, a product $\mu : H \otimes H \to H$, a counit $\varepsilon : H \to \Bbbk$, a coproduct $\Delta : H \to H \otimes H$, and an antipode $S : H \to H$. These structure morphisms are subject to a well-known list of axioms, that the reader can find in \cite[Definitions~III.1.1, III.2.2, \& III.3.2]{K95}. For all elements $x,y \in H$, we will use the short notation $\mu(x \otimes y) = xy$ (for the product), $\eta(1) = 1$ (for the unit), and $\Delta(x) = x_{(1)} \otimes x_{(2)}$ (for the coproduct, which hides a sum).

If $H$ is finite-dimensional, then its \textit{adjoint representation} is the $H$-module given by the $\Bbbk$-vector space $H$ itself, equipped with the adjoint action
\[
 x \triangleright y = x_{(1)}yS(x_{(2)})
\]
for all $x,y \in H$. The $g$th tensor power of the adjoint representation is defined by
\[
 x \triangleright \bfy = (x_{(1)} \triangleright y_1) \otimes \ldots \otimes (x_{(g)} \triangleright y_g)
\]
for all $x \in H$ and $\bfy = y_1 \otimes \ldots \otimes y_g \in H^{\otimes g}$.

A \textit{ribbon structure} on $H$ is given by an R-matrix $R = R'_i \otimes R''_i \in H \otimes H$ (which hides a sum) and by a ribbon element $\vartheta \in H$, see \cite[Definitions~VIII.2.2. \& XIV.6.1]{K95}. We denote by $M \in H \otimes H$ the M-matrix associated with the R-matrix $R$, which is defined by $M = R''_j R'_i \otimes R'_j R''_i$ (with sums hidden).

A left integral $\lambda \in H^*$ of $H$ is a linear form on $H$ satisfying $\lambda(x_{(2)}) x_{(1)} = \lambda(x) 1$ for every $x \in H$, and a left cointegral $\Lambda \in H$ of $H$ is an element of $H$ satisfying $x \Lambda = \varepsilon(x) \Lambda$ for every $x \in H$, see \cite[Definition~10.1.1 \& 10.1.2]{R12}. Recall that, if $H$ is finite-dimensional, then a left integral and a left cointegral exist, they are unique up to scalar, and we can lock together their normalizations by requiring
\[
 \lambda(\Lambda) = 1,
\]
as follows from \cite[Theorem~10.2.2]{R12}.

The Drinfeld map $D : H^* \to H$ of a ribbon Hopf algebra $H$ is the linear map determined by $D(f) := (f \otimes \id_H)(M)$ for every $f \in H^*$, where $M$ is the M-matrix of $H$. By definition, $H$ is \textit{factorizable} if $D$ is a linear isomorphism. This happens if and only if $\lambda(R'_j R''_i)R''_j R'_i$ is a cointegral, see \cite[Theorem~5]{K96} and \cite[Proposition~7.1]{BD21}, and we fix the normalization of both $\lambda$ and $\Lambda$ by asking that
\[
 \lambda(R'_j R''_i)R''_j R'_i = \Lambda.
\]

Let us fix a connected surface $\varSigma_g$ of genus $g$ with one boundary component, that we will represent as follows:
\[
 \pic{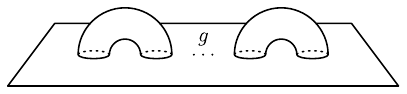}
\]
Let us denote by $\Mod(\varSigma_g)$ the mapping class group of $\varSigma_g$. By definition, it is the group of positive self-diffeomorphisms of $\varSigma_g$ fixing the boundary pointwise, considered up to isotopies fixing the boundary pointwise. As proved in \cite{L64}, it is generated by (positive) Dehn twists
\[
 \{ \twist{\alpha}{j}, \twist{\beta}{j}, \twist{\gamma}{k} \mid 1 \leqs j \leqs g, 1 \leqs k \leqs g-1 \}
\]
along the simple closed curves
\begin{gather*}
 \pic{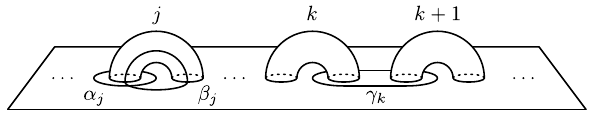}
\end{gather*}
When $g=1$ we set $\tau_\alpha := \tau_{\alpha_1}$ and $\tau_\beta := \tau_{\beta_1}$, and when $g=2$ we set $\tau_\gamma := \tau_{\gamma_1}$.

The \textit{HKL representation} $\rho^H_g : \Mod(\varSigma_g) \to \PGL_H(H^{\otimes g})$, originally defined in \cite[Section~4]{L94}, is determined by
\begin{align}
 \tau_\alpha \cdot x &\propto \vartheta^{-1} x, \label{E:tau_alpha_HKL} \\*
 \tau_\beta \cdot x &\propto \lambda(\vartheta_{(2)} x) S(\vartheta_{(1)}) \label{E:tau_beta_HKL}
\end{align}
for every $x \in H$ when $g = 1$, by 
\begin{align}
 \tau_\gamma \cdot (x_1 \otimes x_2) &\propto x_1 S(\vartheta^{-1}_{(1)}) \otimes \vartheta^{-1}_{(2)} x_2 \label{E:tau_gamma_HKL}
\end{align}
for all $x_1,x_2 \in H$ when $g = 2$, and by 
\begin{align}
 \tau_{\alpha_j} \cdot \bfx &\propto x_1 \otimes \ldots \otimes x_{j-1} \otimes (\tau_\alpha \cdot x_j) \otimes x_{j+1} \otimes \ldots \otimes x_g, \nonumber \\*
 \tau_{\beta_j} \cdot \bfx &\propto x_1 \otimes \ldots \otimes x_{j-1} \otimes (\tau_\beta \cdot x_j) \otimes x_{j+1} \otimes \ldots \otimes x_g, \nonumber \\*
 \tau_{\gamma_k} \cdot \bfx &\propto x_1 \otimes \ldots \otimes x_{k-1} \otimes (\tau_\gamma \cdot (x_k \otimes x_{k+1})) \otimes x_{k+2} \otimes \ldots \otimes x_g \nonumber
\end{align}
for every $\bfx = x_1 \otimes \ldots \otimes x_g \in H^{\otimes g}$ when $g \geqs 2$, as proved in \cite[Proposition~4.2]{DM22}. Here, $\PGL_H(H^{\otimes g})$ denotes the group of invertible $H$-module endomorphisms for (the $g$th tensor power of) the adjoint action of $H$ onto $H^{\otimes g}$, considered up to non-zero scalars in $\Bbbk$.

\section{Small quantum \texorpdfstring{$\fsl_2$}{sl(2)}}\label{S:small_quantum_sl2}

Let us consider a formal variable $q$ and, for all integers $n \in \Z$ and $k,\ell \geqs 0$, let us set
\begin{align*}
 \{ n \}_q &:= q^n-q^{-n}, &
 \{ n;k \}_q &:= \prod_{j=1}^k \{ n-k+j \}_q, &
 \{ k \}_q! &:= \{ k;k \}_q, \\*
 [n]_q &:= \frac{\{ n \}_q}{\{ 1 \}_q}, &
 [k]_q! &:= \prod_{j=1}^k [j]_q, &
 \sqbinom{k}{\ell}_q &:= \frac{[k]_q!}{[\ell]_q![k-\ell]!},
\end{align*}
with the convention that $\{ n;0 \}_q = 1$ and $[0]_q! = 1$. For convenience, we also set
\begin{align*}
 \{ n;k \}_q &:= 0, & 
 \sqbinom{k}{\ell}_q &:= 0
\end{align*}
whenever $k < 0$, $\ell < 0$, or $k < \ell$. Notice that
\begin{align*}
 \{ n \}_q, \{ n;k \}_q, \{ k \}_q!, [n]_q, [k]_q!, \sqbinom{k}{\ell}_q &\in \Z[q,q^{-1}]
\end{align*}
and that
\begin{align*}
 \{ n \}_{q^{-1}} &:= - \{ n \}_q, &
 \{ n;k \}_{q^{-1}} &:= (-1)^k \{ n;k \}_q, &
 \{ k \}_{q^{-1}}! &:= (-1)^k \{ k \}_q!, \\*
 [n]_{q^{-1}} &:= [n]_q, &
 [k]_{q^{-1}}! &:= [k]_q!, &
 \sqbinom{k}{\ell}_{q^{-1}} &:= \sqbinom{k}{\ell}_q
\end{align*}
for all integers $n,k,\ell \in \Z$.

Let us fix an odd integer $r \geqs 3$, and let us consider the primitive $r$th root of unity $\zeta = e^{\frac{2 \pi \fraki}{r}}$. Notice that $[n]_\zeta$, $[k]_\zeta!$, and $\sqbinom{k}{\ell}_\zeta$ are invertible in $\Z[\zeta]$ for all integers $n,k,\ell \in \Z$, see \cite[Lemma~3.1.$(ii)$]{MR97}. For every $n \in \Z/r\Z$, we denote by $G_n$ the Gauss sum
\begin{align}\label{E:Gauss_sum}
 G_n := \sum_{\ell=0}^{r-1} \zeta^{-2\ell(\ell-n)} = \fraki^{\frac{r-1}{2}} \sqrt{r} \zeta^{\frac{r+1}{2}n^2},
\end{align}
see \cite[Appendix~B]{BD21} for a computation.

The \textit{small quantum group} $\fraku_\zeta^\Q = \fraku_\zeta \fsl_2$, first defined by Lusztig in \cite{Lu90}, is the $\Q(\zeta)$-algebra with generators $\{ E,F^{(1)},K \}$ and relations
\begin{gather*}
 E^r = (F^{(1)})^r = 0, \qquad K^r = 1, \\*
 K E K^{-1} = \zeta^2 E, \qquad K F^{(1)} K^{-1} = \zeta^{-2} F^{(1)}, \qquad
 [E,F^{(1)}] = K - K^{-1}.
\end{gather*}
It admits a Hopf algebra structure obtained by setting
\begin{align*}
 \Delta(E) &= E \otimes K + 1 \otimes E, & \varepsilon(E) &= 0, & S(E) &= -E K^{-1}, \\*
 \Delta(F^{(1)}) &= K^{-1} \otimes F^{(1)} + F^{(1)} \otimes 1, & \varepsilon(F^{(1)}) &= 0, & S(F^{(1)}) &= - K F^{(1)}, \\*
 \Delta(K) &= K \otimes K, & \varepsilon(K) &= 1, & S(K) &= K^{-1}.
\end{align*}
Remark that Lusztig considers the opposite coproduct, while we are using the one of Kassel \cite[Section~VII.1]{K95}.

Let us set
\begin{align}\label{E:one_divided_powers_def}
 1_n &:= \frac{1}{r} \sum_{m=0}^{r-1} \zeta^{2mn} K^m, &
 F^{(k)} := \frac{(F^{(1)})^k}{[k]_\zeta!},
\end{align}
for every $n \in \Z/r\Z$ and every integer $0 \leqs k \leqs r-1$. Thanks to \cite[Theorem~5.6]{Lu90}, a basis of $U_g = (\fraku_\zeta^\Q)^{\otimes g}$ over $\Q(\zeta)$ is given by
\begin{equation}\label{E:E1F_basis}
 \calB^{E1F}_g := 
 \left\{ \bfE^{\bfell} \bfone_{\bfm} \bfF^{(\bfn)} \Biggm| 
 \begin{array}{l}
  \bfell = (\ell_1,\ldots,\ell_g), \bfn = (n_1,\ldots,n_g) \in \N^{\times g}, \\
  \bfm = (m_1,\ldots,m_g) \in (\Z/r\Z)^{\times g}, \\
  0 \leqs \ell_j,n_j \leqs r-1 \ \Forall 1 \leqs j \leqs g
 \end{array} \right\},
\end{equation}
where 
\[
 \bfE^{\bfell} \bfone_{\bfm} \bfF^{(\bfn)} := E^{\ell_1} 1_{m_1} F^{(n_1)} \otimes \ldots \otimes E^{\ell_g} 1_{m_g} F^{(n_g)}.
\]
In particular, $\calB^{E1F}_1$ spans a $\Z[\zeta]$-lattice 
\begin{equation}\label{E:quantum_sl2_integral_form}
 \fraku_\zeta^\Z := \langle \calB^{E1F}_1 \rangle_{\Z[\zeta]} \subset \fraku_\zeta^\Q.
\end{equation}

Notice that $\fraku_\zeta^\Z$ is a sub-$\Z[\zeta]$-algebra of $\fraku_\zeta^\Q$, generated by $\{ E,F^{(1)},1_n \mid n \in \Z/r\Z \}$. Indeed, we have
\begin{align}
 1_n E &= E 1_{n+1}, &
 1_n F^{(1)} &= F^{(1)} 1_{n-1}, &
 1_m 1_n &= \delta_{m,n} 1_n
\end{align}
for all $m,n \in \Z/r\Z$, see \cite[Lemma~4.3]{DM22}. Furthermore, $\fraku_\zeta^\Z$ is a Hopf subalgebra of $\fraku_\zeta^\Q$, because
\begin{align*}
 \Delta(1_n) &= \sum_{m=0}^{r-1} 1_{n-m} \otimes 1_m, &
 \varepsilon(1_n) &= \delta_{n,0}, &
 S(1_n) &= 1_{-n}
\end{align*}
for all $m,n \in \Z/r\Z$, see \cite[Lemma~D.1]{DM22}.

Both $\fraku_\zeta^\Z$ and $\fraku_\zeta^\Q$ admit ribbon Hopf algebra structures. Indeed, an \textit{R-matrix} $R = R'_i \otimes R''_i \in \fraku_\zeta^\Z \otimes \fraku_\zeta^\Z$ is given by
\begin{align*}
 R 
 &= \sum_{m,n=0}^{r-1} \zeta^{\frac{n(n-1)}{2}} K^{-m} E^n \otimes 1_m F^{(n)} 
 = \sum_{m,n=0}^{r-1} \zeta^{\frac{n(n-1)}{2}} 1_m E^n \otimes K^{-m} F^{(n)},
\end{align*}
and a ribbon element $\vartheta \in \fraku_\zeta^\Z$ is given by
\begin{align*}
 \vartheta
 &= \sum_{m,n=0}^{r-1} (-1)^m 
 \zeta^{-\frac{(m+3)m}{2}-2(m+n+1)n} E^m F^{(m)} 1_n,
\end{align*}
with inverse $\vartheta^{-1} \in \fraku_\zeta^\Z$ given by
\begin{align*}
 \vartheta^{-1}
 &= \sum_{m,n=0}^{r-1}
 \zeta^{\frac{(m+3)m}{2}+2(m+n+1)n} E^m F^{(m)} 1_n,
\end{align*}
compare with \cite[Example~3.4.3]{M95}, \cite[Lemma~B.2]{BD21}, or \cite[Lemma~D.2]{DM22}.

A non-zero left integral $\lambda$ of $\fraku_\zeta^\Q$ is given by
\begin{equation*}
 \lambda(E^\ell F^{(m)} 1_n) = \frac{\zeta^{-2n}}{\sqrt{r}} \delta_{\ell,r-1} \delta_{m,r-1},
\end{equation*}
and a non-zero two-sided cointegral $\Lambda$ of $\fraku_\zeta^\Q$ satisfying $\lambda(\Lambda) = 1$ is given by
\begin{equation*}
 \Lambda := \sqrt{r} E^{r-1} F^{(r-1)} 1_0,
\end{equation*}
compare with \cite[Proposition~A.5.1]{L94}. The ribbon Hopf algebra $\fraku_\zeta^\Q$ is \textit{factorizable}, as shown in \cite[Corollary~A.3.3]{L94}, see also \cite[Example~3.4.3]{M95}. Then, we set
\begin{align}\label{E:HKL_rep}
 \smash{\rho^\Q_{g,r} := \rho^{\fraku_\zeta^\Q}_g,}
\end{align}
where $\rho^H_g$ is defined by Equations~\eqref{E:tau_alpha_HKL}--\eqref{E:tau_gamma_HKL} for a factorizable ribbon Hopf algebra $H$.

\section{Heisenberg homology groups}\label{S:Heisenberg_homology}

Let us consider the \textit{$n$th unordered configuration space} 
\begin{align*}
 \calC_{n,g} = \calC_n(\varSigma_g) := \{ \{ x_1,\ldots,x_n \} \subset \varSigma_g \mid x_i \neq x_j \ \Forall i \neq j \}
\end{align*}
of the surface $\varSigma_g$. Its fundamental group $\pi_{n,g} = \pi_1(\calC_{n,g})$, called the \textit{$n$th surface braid group}, has been intensively studied, see \cite[Theorem~2.1]{BG05}, based on \cite[Theorem~1.1]{B01}, for an explicit presentation. In particular, $\pi_{n,g}$ admits a quotient\footnote{For all integers $n \geqs 3$ and $g \geqs 1$, the Heisenberg group $\bbH_g$ is isomorphic to the third and largest lower central quotient $\pi_{n,g}/\varGamma_3(\pi_{n,g})$, which is recursively defined by $\varGamma_1(\pi_{n,g}) = \pi_{n,g}$ and by $\varGamma_k(\pi_{n,g}) = [\pi_{n,g},\varGamma_{k-1}(\pi_{n,g})]$ for $k > 1$, see \cite[Theorem~1]{BGG05}.} which is isomorphic to the \textit{$g$th Heisenberg group}
\begin{align*}
 \bbH_g 
 := \left\langle \sigma, \alpha_j, \beta_j \ \Forall 1 \leqs j \leqs g \Biggm| 
 \begin{array}{l}
  {[\sigma,\alpha_j] = [\sigma,\beta_j] = [\alpha_j,\alpha_k] = [\beta_j,\beta_k] = 1,} \\[5pt]
  {[\alpha_j,\beta_k] = \sigma^{-2 \delta_{j,k}} \ \Forall 1 \leqs j,k \leqs g}
 \end{array} \right\rangle
\end{align*}
whenever $n \geqs 2$ (or to a subgroup otherwise), see \cite[Lemma~4.4]{BGG11}. When $g=1$, we set $\alpha := \alpha_1$ and $\beta := \beta_1$. We denote by $\hat{\calC}_{n,g} = \hat{\calC}_n(\varSigma_g)$ the corresponding cover, and by $\hat{p} : \hat{\calC}_{n,g} \to \calC_{n,g}$ its covering map, and we define the \textit{$n$th Heisenberg homology group}\footnote{This definition involves the use of \textit{Borel--Moore homology}, see \cite[Appendix~A.2]{DM22} for more details. See also \cite[Appendix~A.1]{DM22} for a reformulation of the construction of \cite[Section~2.2]{DM22} in terms of regular covers.} as
\begin{align*}
 \calH_{n,g}^\bbH &:= H^\BM_n(\calC_{n,g},\calC_{n,g}^{\partial_-};\Z[\bbH_g])  
 \cong H^\BM_n(\hat{\calC}_{n,g},\hat{p}^{-1}(\calC_{n,g}^{\partial_-})),
\end{align*}
where
\begin{align*}
 \calC_{n,g}^{\partial_-} := \{ \{ x_1,\ldots,x_n \} \in \calC_{n,g} \mid \Exists i \ x_i \in \partial_- \varSigma_g \} \subset \partial \calC_{n,g},
\end{align*}
and where $\partial_- \varSigma_g, \partial_+ \varSigma_g \subset \partial \varSigma_g$ are intervals satisfying $\partial_- \varSigma_g \cup \partial_+ \varSigma_g = \partial \varSigma_g$ and $\partial_- \varSigma_g \cap \partial_+ \varSigma_g = \partial (\partial_- \varSigma_g) = \partial (\partial_+ \varSigma_g)$, that we represent as the bottom side and as the union of the remaining three sides in the picture below, respectively.
\[
 \pic{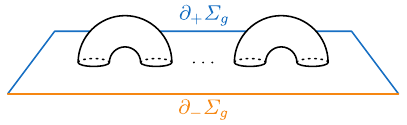}
\]

Twisted homology classes in $\calH_{n,g}^\bbH$ can be represented efficiently through a diagrammatic language based on (systems of) curves in the surface $\varSigma_g$. For instance, the following picture represents a twisted homology class in $\calH_{n,1}^\bbH$:
\[
 \pic{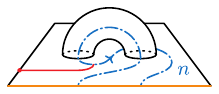}
\]
Indeed, every properly embedded arc $\varGamma \subset \varSigma_g$ satisfying $\partial \varGamma \subset \partial_- \varSigma_g$ determines a homology class $\varGamma \in H_1(\varSigma_g,\partial_- \varSigma_g)$, and its configuration space $\calC_n(\varGamma) \subset \calC_n(\varSigma_g)$ determines a homology class $\varGamma(n) \in H^\BM_n(\calC_{n,g},\calC_{n,g}^{\partial_-})$. The label $n \geqs 0$ in the picture above denotes the switch from an embedded arc in $\varSigma_g$ to its configuration space in $\calC_n(\varSigma_g)$. In order to specify a lift $\hat{\calC}_n(\varGamma)$ of $\calC_n(\varGamma)$ to $\hat{\calC}_n(\varSigma_g)$, we simply need to choose a path from a base-point (which we fix on $\partial_+ \varSigma_g$ for aesthetic purposes) to any point of $\calC_n(\varGamma)$, which yields a twisted homology class $\hat{\varGamma}(n) \in H^\BM_n(\hat{\calC}_{n,g},\hat{p}^{-1}(\calC_{n,g}^{\partial_-}))$. This is what the red path is for (notice that a path in $\calC_{n,g}$ is actually given by an $n$tuple of paths, so we adopt the convention that, whenever an embedded arc is labeled by an integer $n \geqs 0$, its red path actually represents $n$ parallel paths in $\varSigma_g$).

It can be shown that $\calH_{n,g}^\bbH$ is a free $\Z[\bbH_g]$-module of rank $\binom{2g+n-1}{n}$, see \cite[Theorem~A.(a)]{BPS21} or \cite[Proposition~2.9]{DM22}. Furthermore, explicit computation rules \cite[Proposition~2.13]{DM22} allow us to express every twisted homology class in $\calH_{n,g}^\bbH$ as a $\Z[\bbH_g]$-linear combination of twisted homology classes in the basis
\begin{align}\label{E:Gamma_basis}
 \calB^\varGamma_{n,g} := \left\{ \hat{\bfGamma}(\bfa,\bfb) \Biggm| 
 \begin{array}{l}
  \bfa = (a_1,\ldots,a_g), \bfb = (b_1,\ldots,b_g) \in \N^{\times g} \\
  a_1 + b_1 + \ldots + a_g + b_g = n
 \end{array} \right\},
\end{align}
where 
\[
 \hat{\bfGamma}(\bfa,\bfb) := \pic{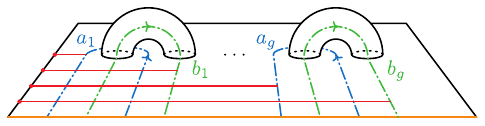}
\]

Every self-diffeomorphism $f$ of $\varSigma_g$ can be extended to a self-diffeomorphism $\calC_n(f)$ of $\calC_n(\varSigma_g)$ by setting
\begin{align*}
 \calC_n(f)\{ x_1,\ldots,x_n \} = \{ f(x_1),\ldots,f(x_n) \},
\end{align*}
and, since $\varGamma_3(\pi_{n,g}) < \pi_{n,g}$ is a characteristic subgroup, $\calC_n(f)$ lifts to a self-diffeomorphism $\hat{\calC}_n(f)$ of $\hat{\calC}_n(\varSigma_g)$. In particular, $\Mod(\varSigma_g)$ acts on the quotient $\bbH_g$ of $\pi_{n,g}$ by setting
\begin{align*}
 f \cdot [\gamma] := [\calC_n(f) \cdot \gamma]
\end{align*}
for every $\gamma \in \pi_{n,g}$, and on the homology group $\calH_{n,g}^\bbH$ by setting
\begin{align*}
 f \cdot \sigma := \hat{\calC}_n(f) \cdot \sigma
\end{align*}
for every $\sigma \in \calH_{n,g}^\bbH$. Notice that this action is \textit{crossed} with respect to the $\Z[\bbH_g]$-module structure on $\calH_{n,g}^\bbH$, in the sense that
\begin{align*}
 f \cdot ([\gamma] \cdot \sigma) = [\calC_n(f) \cdot \gamma] \cdot (\hat{\calC}_n(f) \cdot \sigma)
\end{align*}
for all $\gamma \in \pi_{n,g}$ and $\sigma \in \calH_{n,g}^\bbH$. 

\section{Schrödinger representations}\label{S:Schrdngr}

We consider the \textit{Schrödinger representation}, which is the unique irriducible representation of $\bbH_g$ where $\sigma$ acts by scalar multiplication by $-\zeta^{-2}$. More explicitly, this is obtained by fixing an $r^g$-dimensional $\Q(\zeta)$-vector space $V_g$ with basis 
\begin{align}\label{E:v_basis}
 \calB^v_g := \{ \bfv_{\bfn} = v_{n_1} \otimes \ldots \otimes v_{n_g} \mid \bfn = (n_1,\ldots,n_g) \in (\Z/r\Z)^{\times g} \},
\end{align}
and by considering a homomorphism $\varphi_g : \bbH_g \to \GL_{\Z[\zeta]}(V_g)$ determined by
\begin{align}
 \alpha \cdot v_n &= \zeta^{4n} v_n, \label{E:alpha} \\*
 \beta \cdot v_n &= v_{n+1}, \label{E:beta}
\end{align}
for every $n \in \Z/r\Z$ when $g=1$, and by
\begin{align}
 \alpha_j \cdot \bfv_{\bfn} &= v_{n_1} \otimes \ldots \otimes v_{n_{j-1}} \otimes (\alpha \cdot v_{n_j}) \otimes v_{n_{j+1}} \otimes \ldots \otimes v_{n_g}, \nonumber \\*
 \beta_j \cdot \bfv_{\bfn} &= v_{n_1} \otimes \ldots \otimes v_{n_{j-1}} \otimes (\beta \cdot v_{n_j}) \otimes v_{n_{j+1}} \otimes \ldots \otimes v_{n_g} \nonumber
\end{align}
for every $\bfn = (n_1,\ldots,n_g) \in (\Z/r\Z)^{\times g}$ when $g \geqs 1$, where $\bfv_{\bfn} = v_{n_1} \otimes \ldots \otimes v_{n_g}$. Since $\calH_{n,g}^\bbH$ is a free $\Z[\bbH_g]$-module, changing coefficients amounts to taking a tensor product
\begin{align*}
 \calH_{n,g}^V &:= H^\BM_n(\calC_{n,g},\calC_{n,g}^{\partial_-};V_g)  
 \cong \calH_{n,g}^\bbH \otimes_{\Z[\bbH_g]} V_g.
\end{align*}
The induced action of $\Mod(\varSigma_g)$ on $\calH_{n,g}^V$ can be untwisted by counteracting the effect on coefficients using the projective representation $\psi_g : \Mod(\varSigma_g) \to \PGL_{\Q(\zeta)}(V_g)$ determined by
\begin{align}
 \tau_\alpha \cdot v_n 
 &\propto \frac{1}{G_1} \sum_{\ell=0}^{r-1} \zeta^{-2\ell(\ell-1)} \alpha^\ell \cdot v_n 
 = \frac{1}{G_1} \sum_{\ell=0}^{r-1} \zeta^{-2\ell(\ell-2n-1)} v_n \nonumber \\*
 &= \zeta^{2(n+1)n} v_n, \label{E:tau_alpha_psi} \\*
 \tau_\beta \cdot v_n &\propto \frac{1}{G_1} \sum_{\ell=0}^{r-1} \zeta^{-2(\ell+1)\ell} \beta^\ell \cdot v_n 
 = \frac{1}{G_1} \sum_{\ell=0}^{r-1} \zeta^{-2(\ell+1)\ell} v_{n+\ell} \nonumber \\*
 &= \frac{1}{G_1} \sum_{m=0}^{r-1} \zeta^{-2(m-n+1)(m-n)} v_m \label{E:tau_beta_psi}
\end{align}
for every $n \in \Z/r\Z$ when $g=1$, where $G_1$ is the Gauss sum of Equation~\eqref{E:Gauss_sum}, by
\begin{align}
 \tau_\gamma \cdot (v_{n_1} \otimes v_{n_2}) &\propto \frac{1}{G_1} \sum_{\ell=0}^{r-1} \zeta^{-2(\ell+1)\ell} (\alpha^{-\ell} \cdot v_{n_1}) \otimes (\alpha^\ell \cdot v_{n_2}) \nonumber \\*
 &= \zeta^{2(n_1-n_2+1)(n_1-n_2)} v_{n_1} \otimes v_{n_2} \label{E:tau_gamma_psi}
\end{align}
for $n_1,n_2 \in \Z/r\Z$ when $g=2$, and by
\begin{align}
 \tau_{\alpha_j} \cdot \bfv_{\bfn} 
 &\propto v_{n_1} \otimes \ldots \otimes v_{n_{j-1}} \otimes (\tau_\alpha \cdot v_{n_j}) \otimes v_{n_{j+1}} \otimes \ldots \otimes v_{n_g}, \nonumber \\*
 \tau_{\beta_j} \cdot \bfv_{\bfn} 
 &\propto v_{n_1} \otimes \ldots \otimes v_{n_{j-1}} \otimes (\tau_\beta \cdot v_{n_j}) \otimes v_{n_{j+1}} \otimes \ldots \otimes v_{n_g}, \nonumber \\*
 \tau_{\gamma_k} \cdot \bfv_{\bfn} 
 &\propto v_{n_1} \otimes \ldots \otimes v_{n_{k-1}} \otimes (\tau_\gamma \cdot (v_{n_k} \otimes v_{n_{k+1}})) \otimes v_{n_{k+2}} \otimes \ldots \otimes v_{n_g} \nonumber
\end{align}
for every $\bfn = (n_1,\ldots,n_g) \in (\Z/r\Z)^{\times g}$ when $g \geqs 2$, where $\bfv_{\bfn} = v_{n_1} \otimes \ldots \otimes v_{n_g}$, see \cite[Corollary~2.27]{DM22} (notice that the difference between these formulas and \cite[Equation~(19)]{DM22} is that here we chose representatives in $\PGL_{\Q(\zeta)}(V_g)$ with determinant in $\Z[\zeta]^\times$, by dividing the image of every generator by a factor $G_1$). We obtain a projective action of $\Mod(\varSigma_g)$ onto $\calH_{n,g}^V$ determined by
\begin{align}\label{E:BPS_rep}
 f \cdot (\sigma \otimes v) = (\hat{\calC}_n(f) \cdot \sigma) \otimes (f \cdot v)
\end{align}
for all $f \in \Mod(\varSigma_g)$, $\sigma \in \calH_{n,g}^\bbH$, and $v \in V_g$. 

\section{Isomorphism between quantum and homological representations}\label{S:isomorphism}

Let us set
\begin{align*}
 \calH_{n,g}^{V,\dagger} &:= H_n(\calC_{n,g},\calC_{n,g}^{\partial_-};V_g), &
 \calH_g^{V,\dagger} &:= \bigoplus_{n \geqs 0} \calH_{n,g}^{V,\dagger}.
\end{align*}
Since all standard chains (and cycles) determine Borel--Moore chains (and cycles), we have natural maps $\iota_{n,g} : \calH_{n,g}^{V,\dagger} \to \calH_{n,g}^V$, whose images we denote
\begin{align}
 \calH_{n,g}^{V(r)} &:= \im \iota_{n,g} \subset \calH_{n,g}^V, &
 \calH_g^{V(r)} &:= \bigoplus_{n \geqs 0} \im \iota_{n,g} \subset \bigoplus_{n \geqs 0} \calH_{n,g}^V =: \calH_g^V.
\end{align}

It is shown in \cite[Section~2.4.3]{DM22} that the $\Q(\zeta)$-vector space $\calH_g^{V(r)}$ admits a basis
\begin{align}\label{E:Gamma_v_basis}
 \calB^{\varGamma,v}_g := 
 \left\{ \hat{\bfGamma}(\bfa,\bfb) \otimes \bfv_{\bfn} \Biggm| 
 \begin{array}{l}
  \bfa = (a_1,\ldots,a_g), \bfb = (b_1,\ldots,b_g) \in \N^{\times g}, \\
  \bfn = (n_1,\ldots,n_g) \in (\Z/r\Z)^{\times g}, \\
  0 \leqs a_j,b_j \leqs r-1 \ \Forall 1 \leqs j \leqs g
 \end{array} \right\},
\end{align}
built from the bases $\calB^\varGamma_{n,g}$ and $\calB^v_g$ defined by Equations~\eqref{E:Gamma_basis} and \eqref{E:v_basis} respectively.

As explained in \cite[Section~2.3]{DM22}, the subspace $\calH_g^{V(r)}$ of $\calH_g^V$ can be equipped with the structure of a $\fraku_\zeta^\Z$-module. Roughly speaking, the action of $E$ is given by (the restriction to $\calH_{n,g}^{V(r)}$ of) the operator
\begin{align}
 \calE : \calH_{n,g}^V &\to H^\BM_{n-1}(\calC_{n,g}^{\partial_-},\calC_{n,g}^{2 \partial_-};V_g) \cong \calH_{n-1,g}^V \label{E:homological_E} \\*
 \sigma \otimes v &\mapsto (-1)^{n-1} \partial_*(\sigma) \otimes v, \nonumber
\end{align}
where $\partial_*$ is the connection homomorphism in the long exact sequence of the triple $(\calC_{n,g},\calC_{n,g}^{\partial_-},\calC_{n,g}^{2 \partial_-})$ for 
\begin{align*}
 \calC_{n,g}^{2 \partial_-} := \{ \{ x_1,\ldots,x_n \} \in \calC_{n,g} \mid \Exists i < j \ x_i,x_j \in \partial_- \varSigma_g \} \subset \calC_{n,g}^{\partial_-},
\end{align*}
see \cite[Definition~2.14]{DM22} for more details. Similarly, for every $0 \leqs k \leqs r-1$, the action of $F^{(k)}$ is given by (the restriction to $\calH_{n,g}^{V(r)}$ of) the operator
\begin{align}
 \calF^{(k)} : \calH_{n,g}^V &\to \calH_{n+k,g}^V \label{E:homological_F} \\*
 \sigma \otimes v &\mapsto \zeta^{\frac{k(k-1)}{2}+2kg} (\hat{\Phi}(k) \times \sigma) \otimes v, \nonumber
\end{align}
where $\hat{\Phi}(k) \times \sigma \in \calH_{n+k,g}^\bbH$ is the twisted homology class obtained from $\sigma \in \calH_{n,g}^\bbH$ by gluing a collar $\partial_+ \varSigma_g \times [0,1]$ of $\partial_+ \varSigma_g$ to $\varSigma_g$ whose core $\Phi = \partial_+ \varSigma_g \times \{ \frac{1}{2} \}$ is labeled by the integer $k$, see \cite[Definition~2.15]{DM22} for more details. Finally, for every $m \in \Z/r\Z$, the action of $1_m$ is given by (the restriction to $\calH_{n,g}^{V(r)}$ of) the projector
\begin{align}
 1_m : \calH_{n,g}^V &\to \calH_{n,g}^V \label{E:homological_1} \\*
 \sigma \otimes v &\mapsto \delta_{m,n+g}(r) \sigma \otimes v, \nonumber
\end{align}
where $\delta_{m,n+g}(r) = 1$ if $m \equiv n+g \pmod r$, and $\delta_{m,n+g}(r) = 0$ otherwise, as follows from \cite[Definition~2.16]{DM22} and Equation~\eqref{E:one_divided_powers_def}. Notice that this action of $\fraku_\zeta^\Z$ clearly commutes with the projective action of $\Mod(\varSigma_g)$, since the latter is induced by diffeomorphisms that fix the boundary of $\varSigma_g$ pointwise, while the former only involves boundary operations.

For all $\bfa=(a_1,\ldots,a_g), \bfb = (b_1,\ldots,b_g) \in \Z^{\times g}$ and $\bfn = (n_1,\ldots,n_r) \in (\Z/r\Z)^{\times g}$, let us set
\begin{align*}
 N_k(a_k,b_k,n_k) &:= \zeta^{2(a_k+b_k)(k-1)+\frac{a_k(a_k-1)}{2}+2a_kb_k-2(b_k-1)n_k} \in \Z[\zeta], \\*
 N(\bfa,\bfb,\bfn) &:= \prod_{1 \leqs i < j \leqs g} \zeta^{2(a_i+b_i)(a_j+b_j)} \prod_{k=1}^g N_k(a_k,b_k,n_k) \in \Z[\zeta].
\end{align*}
It is shown in \cite[Theorem~6.1]{DM22} that, if we set
\begin{align*}
 \bar{\bfk} &= (k_g,\ldots,k_1), &
 \iota(\bfk) &= (r-k_1-1,\ldots,r-k_g-1)
\end{align*}
for every $\bfk = (k_1,\ldots,k_g) \in \Z^{\times g}$, then
\begin{align}
 \Phi_g : \calH_g^{V(r)} &\to U_g \label{E:isomorphism} \\*
 \hat{\bfGamma}(\bfa,\bfb) \otimes \bfv_{\bfn} &\mapsto N(\bfa,\bfb,\bfn) \bfE^{\iota(\bar{\bfb})} \bfone_{\bar{\bfn}} \bfF^{(\bar{\bfa})} \nonumber
\end{align}
defines an isomorphism of representations of both $\fraku_\zeta^\Z$ and $\Mod(\varSigma_g)$. 

\begin{remark}\label{R:asymptotic_faithfulness}
 An important byproduct of the isomorphism $\Phi_g : \calH_g^{V(r)} \to U_g$ defined by Equation~\eqref{E:isomorphism} is that the projective action of $\Mod(\varSigma_g)$ onto $U_g$ is asymptotically faithful. Indeed, on the one hand, if
 \begin{align*}
  f \in \bigcap_{\substack{r \geqs 3 \\ \mathclap{r \equiv 1 \pmod 2}}} \ker \rho_{g,r}^\Q,
 \end{align*}
 then, in particular, $\hat{\calC}_n(f) \cdot \sigma = \sigma$ for all $n \geqs 0$ and $\sigma \in \calH_{n,g}^\bbH$, thanks to Equation~\eqref{E:BPS_rep}. On the other hand, thanks to \cite[Proposition~40]{BPS21}, the kernel of the action of $\Mod(\varSigma_g)$ onto $\calH_{n,g}^\bbH$ is contained in the $n$th term $J_n(\Mod(\varSigma_g))$ of the Johnson filtration of $\Mod(\varSigma_g)$. Since their total intersection is trivial, we have 
 \begin{align*}
  f \in \bigcap_{n \geqs 0} J_n(\Mod(\varSigma_g)) = \{ \id \}.
 \end{align*}
\end{remark}

\begin{remark}\label{R:integrality}
 Another important consequence of the isomorphism $\Phi_g : \calH_g^{V(r)} \to U_g$ defined by Equation~\eqref{E:isomorphism} is that it allows us to reduce the question of integrality for $\rho_{g,r}^\Q : \Mod(\varSigma_g) \to \PGL_{\Q(\zeta)}(U_g)$ to the same one for $\psi_g : \Mod(\varSigma_g) \to \PGL_{\Q(\zeta)}(V_g)$, thanks to Equation~\eqref{E:BPS_rep}. Indeed, the action of $\Mod(\varSigma_g)$ onto $\calH_{n,g}^\bbH$ is a (crossed) $\Z[\bbH_g]$-linear action, so it is integral by definition, and the Schrödinger representation $\varphi_g$ only features coefficients in $\Z[\zeta]$. Therefore, $\psi_g : \Mod(\varSigma_g) \to \PGL_{\Q(\zeta)}(V_g)$ is the only ingredient that features coefficients in $\Q(\zeta)$, so the homological model naturally localizes the part of the construction that poses problems for what concerns integrality.
\end{remark}

\section{Triangular basis}\label{S:triangular_basis}

In this section, we define a basis 
\begin{align}\label{E:t_basis}
 \calB^t_g &:= \{ \bft_{\bfn} = t_{n_1} \otimes \ldots \otimes t_{n_g} \mid 0 \leqs n_j \leqs r-1 \ \Forall 1 \leqs j \leqs g  \}
\end{align}
of the Schrödinger representation $V_g$, which we refer to as the \textit{triangular basis}. For every integer $0 \leqs n \leqs r-1$, let us set
\begin{align}\label{E:t_to_v}
 t_n &:= \sum_{k=0}^n (-1)^{n-k} \zeta^{(n-k)(n-1)} \sqbinom{n}{k}_\zeta v_k.
\end{align}
Clearly, $\calB^t_1$ is a basis of $V$. Indeed, the transition matrix from $\calB^t_1$ to $\calB^v_1$ is upper-triangular, with constant diagonal $1$. The next result then gives the inverse transition matrix, from $\calB^v_1$ back to $\calB^t_1$.

\begin{lemma}\label{L:v_to_t}
 For all integers $0 \leqs n \leqs r-1$ we have
 \begin{align}\label{E:v_to_t}
  v_n &= \sum_{k=0}^n \zeta^{k(n-k)} \sqbinom{n}{k}_\zeta t_k.
 \end{align}
\end{lemma}

For a proof of Lemma~\ref{L:v_to_t}, see Appendix~\ref{A:computations}. The following statement motivates the name of $\calB^t_g$: indeed, these bases make $\alpha_j \in \bbH_g$, $\tau_{\alpha_j}, \tau_{\gamma_k} \in \Mod(\varSigma_g)$ into upper-triangular matrices, while they make $\beta_j \in \bbH_g$ and $\tau_{\beta_j} \in \Mod(\varSigma_g)$ into lower-triangular ones for all integers $1 \leqs j \leqs g$ and $1 \leqs k \leqs g-1$.

\begin{proposition}\label{P:triangularity}
 For every integer $0 \leqs n \leqs r-1$ we have
 \begin{align*}
  \alpha \cdot t_n 
  &= \zeta^{4n} \sum_{k=0}^2 \zeta^{-\frac{k(k-2n+5)}{2}} \{ 2;k \}_\zeta \sqbinom{n}{k}_\zeta t_{n-k}, \\
  \beta \cdot t_n 
  &= \zeta^{2n} \sum_{k=0}^1 \zeta^{-k(k+2n-1)} \sqbinom{1}{k}_\zeta t_{n+k},
 \end{align*}
 where $t_r := 0$. Furthermore, for all integers $0 \leqs n,n_1,n_2 \leqs r-1$, we have
 \begin{align*}
  \tau_\alpha \cdot t_n 
  &\propto \zeta^{2(n+1)n} \sum_{k=0}^n (-1)^k \zeta^{k(2k-3n-3)} \sqbinom{n}{k}_\zeta C_{4(n-k)+3,k}(\zeta) t_{n-k}, \\
  \tau_\beta \cdot t_n 
  &\propto (-1)^n \frac{\zeta^{-(n+3)n}}{G_1} \sum_{k=0}^{r-n-1} \zeta^{-2k(k+2n+1)} D_{r-n-k-1,n}(\zeta) t_{n+k}, \\
  \tau_\gamma \cdot (t_{n_1} \otimes t_{n_2})
  &\propto \zeta^{2(n_1-n_2+1)(n_1-n_2)} \sum_{k_1=0}^{n_1} \sum_{k_2=0}^{n_2} (-1)^{k_1+k_2} \\*
  &\hspace*{\parindent} \zeta^{2(k_1-k_2)^2-k_1(3n_1-4n_2+3)+k_2(4n_1-3n_2+1)} \sqbinom{n_1}{k_1}_\zeta \sqbinom{n_2}{k_2}_\zeta \\*
  &\hspace*{\parindent} E_{4(n_1-k_1-n_2+k_2)+3,-4(n_1-k_1-n_2+k_2)-1,k_1,k_2}(\zeta) t_{n_1-k_1} \otimes t_{n_2-k_2},
 \end{align*}
 where $C_{m,n}(\zeta), D_{m,n}(\zeta), E_{m_1,m_2,n_1,n_1}(\zeta) \in \Z[\zeta]$ are defined by Equations~\eqref{E:C_def}, \eqref{E:D_def}, and \eqref{E:E_def} for all $m,m_1,m_2 \in \Z$, and where $G_n \in \Z[\zeta]$ is defined by Equation~\eqref{E:Gauss_sum}.
\end{proposition}

Since the proof of Proposition~\ref{P:triangularity} is a direct computation, we postpone it to Appendix~\ref{A:computations}. 

\section{Integral rescaling}\label{S:integral_rescaling}

In this section, we define a basis 
\begin{align}\label{E:w_basis}
 \calB^{v'}_g &:= \{ \bfv'_{\bfn} = v'_{n_1} \otimes \ldots \otimes v'_{n_g} \mid 0 \leqs n_j \leqs r-1 \ \Forall 1 \leqs j \leqs g \}
\end{align}
of the Schrödinger representation $V_g$ obtained from $\calB^t_g$ by rescaling, which we refer to as the \textit{integral basis}. For every $\bfn = (n_1,\ldots,n_g) \in \Z^{\times g}$ satisfying $0 \leqs n_j \leqs r-1$ for every integer $1 \leqs j \leqs g$, let us set
\begin{align}\label{E:w_to_t}
 \bfv'_{\bfn} &:= h(\zeta)^{- \left\lfloor \frac{\sum_{j=1}^g n_j}{2} \right\rfloor} \bft_{\bfn},
\end{align}
where
\begin{align}
 h(q) &:= 1-q \in \Z[q,q^{-1}]. \label{E:h_def}
\end{align}

In order to motivate the name of $\calB^{v'}_g$, let us recall a few simple results from elementary number theory. If $x,y \in \Z[\zeta]$, then we write $x \sim y$ whenever there exists an invertible $z \in \Z[\zeta]^\times$ such that $x = yz$. Then, we have
\begin{align}\label{E:h_div}
 h(\zeta)^{\frac{r-1}{2}} \sim G_n
\end{align} 
for every $n \in \Z/r\Z$, where $G_n$ is the Gauss sum defined by Equation~\eqref{E:Gauss_sum}. Indeed, $h(\zeta)^{r-1} \sim r \sim G_n^2$, but $h(\zeta)$ is prime in $\Z[\zeta]$, as proved in \cite[Lemma~3.1.$(i)$ \& $(iii)$]{MR97}, so Equation~\eqref{E:h_div} follows.

\begin{proposition}\label{P:divisibility}
 For all integers $m,n \geqs 0$, we have
 \begin{align}
  h(q)^{\left\lfloor \frac{n+1}{2} \right\rfloor} &\mid C_{m,n}(q), \label{E:C_div} \\*
  h(q)^{\left\lfloor \frac{m+n+1}{2} \right\rfloor} &\mid D_{m,n}(q), \label{E:D_div} \\*
  h(q)^{\left\lfloor \frac{n_1+n_2+1}{2} \right\rfloor} &\mid E_{m_1,m_2,n_1,n_2}(q), \label{E:E_div}
 \end{align}
 where $C_{m,n}(q), D_{m,n}(q), E_{m_1,m_2,n_1,n_1}(q) \in \Z[q,q^{-1}]$ are defined by Equations~\eqref{E:C_def}, \eqref{E:D_def}, and \eqref{E:E_def}.
\end{proposition}

Since, once again, the proof of Proposition~\ref{P:divisibility} is obtained by direct computation, we postpone it to Appendix~\ref{A:divisibility}.

\section{Main theorem}\label{S:main_theorem}

In this section, we prove our main result.

\begin{theorem}\label{T:integral_Schrdngr}
 If $r$ is prime, then the $\Z[\zeta]$-lattice $V'_g = \langle \calB^{v'}_g \rangle_{\Z[\zeta]} \subset V_g$ defined by Equation~\eqref{E:w_basis} is invariant under the actions of both $\bbH_g$ and $\Mod(\varSigma_g)$.
\end{theorem}

\begin{proof}
 Thanks to Proposition~\ref{P:triangularity}, we have
 \begin{align*}
  \alpha \cdot v'_n 
  &= \zeta^{4n} \sum_{k=0}^2 \zeta^{-\frac{k(k-2n+5)}{2}} \frac{\{ 2;k \}_\zeta}{h(\zeta)^{\left\lfloor \frac{n}{2} \right\rfloor - \left\lfloor \frac{n-k}{2} \right\rfloor}} \sqbinom{n}{k}_\zeta v'_{n-k}, \\
  \beta \cdot v'_n 
  &= \zeta^{2n} \sum_{k=0}^1 \zeta^{-k(k+2n-1)} h(\zeta)^{\left\lfloor \frac{n+k}{2} \right\rfloor - \left\lfloor \frac{n}{2} \right\rfloor} \sqbinom{1}{k}_\zeta v'_{n+k},
 \end{align*}
 where $t_r := 0$. Furthermore, we have
 \begin{align*}
  \tau_\alpha \cdot v'_n 
  &\propto \zeta^{2(n+1)n} \sum_{k=0}^n (-1)^k \zeta^{k(2k-3n-3)} \sqbinom{n}{k}_\zeta 
  \frac{C_{4n-4k+3,k}(\zeta)}{h(\zeta)^{\left\lfloor \frac{n}{2} \right\rfloor - \left\lfloor \frac{n-k}{2} \right\rfloor}} v'_{n-k}, \\
  \tau_\beta \cdot v'_n 
  &\propto (-1)^n \zeta^{-(n+3)n} \sum_{k=0}^{\mathclap{r-n-1}} \zeta^{-2k(k+2n+1)} \\*
  &\hspace*{\parindent} \frac{h(\zeta)^{\left\lfloor \frac{n+k}{2} \right\rfloor - \left\lfloor \frac{n}{2} \right\rfloor} D_{r-n-k-1,n}(\zeta)}{G_1} v'_{n+k}, \\
  \tau_\gamma \cdot (v'_{n_1} \otimes v'_{n_2})
  &\propto \zeta^{2(n_1-n_2+1)(n_1-n_2)} \sum_{k_1=0}^{n_1} \sum_{k_2=0}^{n_2} (-1)^{k_1+k_2} \\*
  &\hspace*{\parindent} \zeta^{2(k_1-k_2)^2-k_1(3n_1-4n_2+3)+k_2(4n_1-3n_2+1)} \sqbinom{n_1}{k_1}_\zeta \sqbinom{n_2}{k_2}_\zeta \\*
  &\hspace*{\parindent} \frac{E_{4(n_1-k_1-n_2+k_2)+3,-4(n_1-k_1-n_2+k_2)-1,k_1,k_2}(\zeta)}{h(\zeta)^{\left\lfloor \frac{n_1+n_2}{2} \right\rfloor - \left\lfloor \frac{n_1+n_2-k_1-k_2}{2} \right\rfloor}} v'_{n_1-k_1} \otimes v'_{n_2-k_2},
 \end{align*}
 Then, the claim follows from the identities
 \begin{align*}
  \left\lfloor \frac{n}{2} \right\rfloor - \left\lfloor \frac{n-k}{2} \right\rfloor &\leqs \left\lfloor \frac{k+1}{2} \right\rfloor, \\
  \left\lfloor \frac{n+k}{2} \right\rfloor - \left\lfloor \frac{n}{2} \right\rfloor + \left\lfloor \frac{r-k}{2} \right\rfloor &\geqs \frac{r-1}{2}, \\
  \left\lfloor \frac{n_1+n_2}{2} \right\rfloor - \left\lfloor \frac{n_1+n_2-k_1-k_2}{2} \right\rfloor &\leqs \left\lfloor \frac{k_1+k_2+1}{2} \right\rfloor
 \end{align*}
 together with Equations~\eqref{E:C_div}, \eqref{E:D_div}, and \eqref{E:E_div}.
\end{proof}

\begin{remark}
 In the proof of Theorem~\ref{T:integral_Schrdngr} we only checked integrality of coefficients for the action of positive Dehn twist generators $\tau_\alpha$, $\tau_\beta$, $\tau_\gamma$, because the same property for their inverses is an immediate consequence of our choice of representatives with determinant in $\Z[\zeta]^\times$ for $\psi_g(\tau_\alpha), \psi_g(\tau_\beta), \psi_g(\tau_\gamma) \in \PGL_{\Q(\zeta)}(V_g)$, see Equations~\eqref{E:tau_alpha_psi}--\eqref{E:tau_gamma_psi}. Notice that the same implication does not hold for the choice of representatives appearing in \cite[Equation~(19)]{DM22}, as observed in \cite[Remark~2.25]{DM22}.
\end{remark}

Let us consider the basis of $\calH_g^{V(r)}$ defined by
\begin{align}\label{E:Gamma_w_basis}
 \calB^{\varGamma,v'}_g :=
 \left\{ \hat{\bfGamma}(\bfa,\bfb) \otimes \bfv'_{\bfn} \Biggm| 
 \begin{array}{l}
  \bfa = (a_1,\ldots,a_g), \bfb = (b_1,\ldots,b_g) \in \N^{\times g}, \\
  \bfn = (n_1,\ldots,n_g) \in \N^{\times g}, \\
  0 \leqs a_j,b_j,n_j \leqs r-1 \ \Forall 1 \leqs j \leqs g
 \end{array} \right\},
\end{align}
which is built from the bases $\calB^\varGamma_{n,g}$ and $\calB^{v'}_g$ defined by Equations~\eqref{E:Gamma_basis} and \eqref{E:w_basis} respectively. Let us also denote by
\begin{equation}\label{E:E1primeF_basis}
 \calB^{E1'F}_g := 
 \left\{ \bfE^{\bfell} \bfone'_{\bfm} \bfF^{(\bfn)} \Biggm| 
 \begin{array}{l}
  \bfell = (\ell_1,\ldots,\ell_g), \bfn = (n_1,\ldots,n_g) \in \N^{\times g}, \\
  \bfm = (m_1,\ldots,m_g) \in \N^{\times g}, \\
  0 \leqs \ell_j,m_j,n_j \leqs r-1 \ \Forall 1 \leqs j \leqs g
 \end{array} \right\}
\end{equation}
the basis of $U_g$ defined by
\[
 \bfE^{\bfell} \bfone'_{\bfm} \bfF^{(\bfn)} := h(\zeta)^{- \left\lfloor \frac{\sum_{j=1}^g m_j}{2} \right\rfloor} E^{\ell_1} T_{m_1} F^{(n_1)} \otimes \ldots \otimes E^{\ell_g} T_{m_g} F^{(n_g)},
\]
where
\[
 T_n := \sum_{k=0}^n (-1)^{n-k} \zeta^{(n-k)(n-1)} \sqbinom{n}{k}_\zeta 1_k
\]
for every integer $0 \leqs n \leqs r-1$. Finally, let us denote by 
\begin{align}\label{E:lattices}
 \calH_g^{V'(r)} &:= \langle \calB^{\varGamma,v'}_g \rangle_{\Z[\zeta]} \subset \calH_g^{V(r)}, &
 U'_g &:= \langle \calB^{E1'F}_g \rangle_{\Z[\zeta]} \subset U_g
\end{align}
the $\Z[\zeta]$-lattices spanned by these two bases.

\begin{corollary}\label{C:main_result}
 The $\Q(\zeta)$-linear isomorphism $\Phi_g : \calH_g^{V(r)} \to U_g$ defined by Equation~\eqref{E:isomorphism} restricts to a $\Z[\zeta]$-linear isomorphism $\Phi'_g : \calH_g^{V'(r)} \to U'_g$ that intertwines the actions of both $\fraku_\zeta^\Z$ and $\Mod(\varSigma_g)$.
\end{corollary}

\begin{proof}
 Up to reordering, the basis $\calB^{E1'F}_g$ of Equation~\eqref{E:E1primeF_basis} is obtained from the image of the basis $\calB^{\varGamma,v'}_g$ of Equation~\eqref{E:Gamma_w_basis} under the isomorphism $\Phi_g : \calH_g^{V(r)} \to U_g$ of Equation~\eqref{E:isomorphism} by a diagonal rescaling, with each rescaling factor being a power of $\zeta$. Therefore, the two bases span the same $\Z[\zeta]$-lattice $U'_g \subset U_g$. Then, the claim about the action of $\Mod(\varSigma_g)$ follows directly from Theorem~\ref{T:integral_Schrdngr}. Furthermore, since the action of $E,F^{(k)},1_m \in \fraku_\zeta^\Z$ on $\sigma \otimes v \in \calH_{n,g}^{V(r)} \subset \calH_{n,g}^\bbH \otimes_{\Z[\bbH_g]} V_g$ determined by Equations~\eqref{E:homological_E}--\eqref{E:homological_1} is clearly independent of $v \in V_g$, then the $\Z[\zeta]$-lattices $\calH_{n,g}^{V'(r)} \subset \calH_{n,g}^{V(r)}$ and $U'_g \subset U_g$ are both stable also under the action of $\fraku_\zeta^\Z$.
\end{proof}

\appendix

\section{Quantum identities}\label{A:quantum_identities}

We will make extensive use of the identities
\begin{align}
 \sqbinom{n}{k}_q &= q^k \sqbinom{n-1}{k}_q + q^{k-n} \sqbinom{n-1}{k-1}_q, \label{E:Tartaglia_1} \\*
 \sqbinom{n}{k}_q &= q^{-k} \sqbinom{n-1}{k}_q + q^{n-k} \sqbinom{n-1}{k-1}_q, \label{E:Tartaglia_2} 
\end{align}
which hold for all integers $n > 0$ and $k \in \Z$, see \cite[Equations~(1) \& (2), Section~0.2]{J96}.

For all integers $\ell,m,n \geqs 0$ satisfying $m \leqs n+\ell$, let us set
\begin{align}\label{E:A_def}
 A_{\ell,m,n}(q) &:= \sum_{k=m-\ell}^n (-1)^k q^{k(m-n+1)} \sqbinom{n}{k}_q \sqbinom{k+\ell}{m}_q.
\end{align}

\begin{lemma}
 For all integers $\ell,m \geqs 0$ and $n > 0$ satisfying $m \leqs n+\ell$, the polynomial $A_{\ell,m,n}(q) \in \Z[q,q^{-1}]$ satisfies the recurrence relation
 \begin{align}\label{E:A_rec}
  A_{\ell,m,n}(q) &= A_{\ell,m,n-1}(q) - q^{m-2n+2} A_{\ell+1,m,n-1}(q).
 \end{align}
\end{lemma}

\begin{proof}
 We have
 \begin{align*}
  A_{\ell,m,n}(q)\
  &\stackrel{\mathclap{\eqref{E:A_def}}}{=} \ \sum_{k=m-\ell}^n (-1)^k q^{k(m-n+1)} \sqbinom{n}{k}_q \sqbinom{k+\ell}{m}_q \\
  &\stackrel{\mathclap{\eqref{E:Tartaglia_1}}}{=} \ \sum_{k=m-\ell}^n (-1)^k q^{k(m-n+1)} \left( q^k \sqbinom{n-1}{k}_q + q^{k-n} \sqbinom{n-1}{k-1}_q \right) \sqbinom{k+\ell}{m}_q \\
  &= \ \sum_{k=m-\ell}^n (-1)^k q^{k(m-n+2)} \sqbinom{n-1}{k}_q \sqbinom{k+\ell}{m}_q \\*
  &\hspace*{\parindent} \ - q^{m-2n+2} \sum_{k=m-\ell}^n (-1)^{k-1} q^{(k-1)(m-n+2)} \sqbinom{n-1}{k-1}_q \sqbinom{k+\ell}{m}_q \\
  &\stackrel{\mathclap{\eqref{E:A_def}}}{=} \ A_{\ell,m,n-1}(q) 
  - q^{m-2n+2} A_{\ell+1,m,n-1}(q). \qedhere
 \end{align*}
\end{proof}

\begin{lemma}
 For all integers $\ell,m,n \geqs 0$ satisfying $m \leqs n+\ell$, we have
 \begin{align}\label{E:A_prop}
  A_{\ell,m,n}(q) = (-1)^n q^{(\ell+1)n} \sqbinom{\ell}{m-n} _q.
 \end{align}
\end{lemma}

\begin{proof}
 Let us prove Equation~\eqref{E:A_prop} by induction on $n \geqs 0$ for all $0 \leqs m \leqs n+\ell$. If $n = 0$ and $0 \leqs m \leqs \ell$, then we have 
 \begin{align*}
  A_{\ell,m,0}(q) \
  &\stackrel{\mathclap{\eqref{E:A_def}}}{=} \ \sum_{k=m-\ell}^0 (-1)^k q^{k(m+1)} \sqbinom{0}{k}_q \sqbinom{k+\ell}{m}_q 
  = \ \sqbinom{\ell}{m}_q.
 \end{align*}
 If $n > 0$ and $0 \leqs m \leqs n+\ell$, then we have
 \begin{align*}
  A_{\ell,m,n}(q) \
  &\stackrel{\mathclap{\eqref{E:A_rec}}}{=} \ A_{\ell,m,n-1}(q) 
  - q^{m-2n+2} A_{\ell+1,m,n-1}(q) \\
  &= \ (-1)^{n-1} q^{(\ell+1)(n-1)} \sqbinom{\ell}{m-n+1}_q 
  - (-1)^{n-1} q^{m+\ell(n-1)} \sqbinom{\ell+1}{m-n+1}_q \\
  &= \ (-1)^n q^{(\ell+1)n} \left( - q^{-\ell-1} \sqbinom{\ell}{m-n+1}_q + q^{m-\ell-n} \sqbinom{\ell+1}{m-n+1}_q \right) \\*
  &\stackrel{\mathclap{\eqref{E:Tartaglia_2}}}{=} \ (-1)^n q^{(\ell+1)n} \sqbinom{\ell}{m-n}_q. \qedhere
 \end{align*}
\end{proof}

For every integer $n \geqs 0$, let us set
\begin{align}\label{E:B_def}
 B_n(q) &:= \sum_{k=0}^n (-1)^k q^{-k(n-5)} \sqbinom{n}{k}_q.
\end{align}

\begin{lemma}
 For every integer $n > 0$, the polynomial $B_n(q) \in \Z[q,q^{-1}]$ satisfies the recurrence relation
 \begin{align}\label{E:B_rec}
  B_n(q) = q^{-n+3} \{ n-3 \}_q B_{n-1}(q).
 \end{align}
\end{lemma}

\begin{proof}
 We have 
 \begin{align*}
  B_n(q) \
  &\stackrel{\mathclap{\eqref{E:B_def}}}{=} \ \sum_{k=0}^n (-1)^k q^{-k(n-5)} \sqbinom{n}{k}_q \\
  &\stackrel{\mathclap{\eqref{E:Tartaglia_1}}}{=} \ \sum_{k=0}^n (-1)^k q^{-k(n-5)} \left( q^k \sqbinom{n-1}{k}_q + q^{k-n} \sqbinom{n-1}{k-1}_q \right) \\
  &= \ \sum_{k=0}^n (-1)^k q^{-k(n-6)} \sqbinom{n-1}{k}_q
  - q^{-2(n-3)} \sum_{k=0}^n (-1)^{k-1} q^{-(k-1)(n-6)} \sqbinom{n-1}{k-1}_q \\
  &\stackrel{\mathclap{\eqref{E:B_def}}}{=} \ B_{n-1}(q) - q^{-2(n-3)} B_{n-1}(q). \qedhere
 \end{align*}
\end{proof}

\begin{lemma}
 For every integer $n \geqs 0$, we have 
 \begin{align}\label{E:B_prop}
  B_n(q) = (-1)^n q^{-\frac{n(n-5)}{2}} \{ 2;n \}_q.
 \end{align}
\end{lemma}

\begin{proof}
 Let us prove Equation~\eqref{E:B_prop} by induction on $n \geqs 0$. If $n = 0$, then we have
 \begin{align*}
  B_0(q) \
  &\stackrel{\mathclap{\eqref{E:B_def}}}{=} \ \sum_{k=0}^0 (-1)^k q^{5k} \sqbinom{0}{k}_q
  = 1.
 \end{align*}
 If $n > 0$, then we have 
 \begin{align*}
  B_n(q) \
  &\stackrel{\mathclap{\eqref{E:B_rec}}}{=} \ q^{-n+3} \{ n-3 \}_q B_{n-1}(q) \\
  &\stackrel{\mathclap{\eqref{E:B_def}}}{=} \ q^{-n+3} \{ n-3 \}_q (-1)^{n-1} q^{-\frac{(n-1)(n-6)}{2}} \{ 2;n-1 \}_q \\
  &= \ (-1)^n q^{-\frac{n(n-5)}{2}} \{ 2;n \}_q. \qedhere
 \end{align*}
\end{proof}

For all $m,n \in \Z$ with $n \geqs 0$, let us set
\begin{align} \label{E:C_def}
 C_{m,n}(q) &:= \sum_{k=0}^n (-1)^k q^{k(2k+m-n)} \sqbinom{n}{k}_q.
\end{align}

\begin{lemma}
 For all $m,n \in \Z$ with $n \geqs 0$, the polynomial $C_{m,n}(q) \in \Z[q,q^{-1}]$ satisfies the recurrence relation
 \begin{align} \label{E:C_prop}
  C_{m,n}(q) &= C_{m+2,n+1}(q) + q^{m+3} C_{m+4,n}(q).
 \end{align}
\end{lemma}

\begin{proof}
 We have
 \begin{align*}
  C_{m,n}(q) \
  &\stackrel{\mathclap{\eqref{E:C_def}}}{=} \ \sum_{k=0}^n (-1)^k q^{k(2k+m-n)} \sqbinom{n}{k}_q \\
  &\stackrel{\mathclap{\eqref{E:Tartaglia_2}}}{=} \ \sum_{k=0}^n (-1)^k q^{k(2k+m-n)} \left( q^k \sqbinom{n+1}{k}_q - q^{n+1} \sqbinom{n}{k-1}_q \right) \\
  &= \ \sum_{k=0}^n (-1)^k q^{k(2k+m-n+1)} \sqbinom{n+1}{k}_q \\*
  &\hspace*{\parindent} \ + q^{m+3} \sum_{k=0}^n (-1)^{k-1} q^{(k-1)(2k+m-n+2)} \sqbinom{n}{k-1}_q \\
  &\stackrel{\mathclap{\eqref{E:C_def}}}{=} \ C_{m+2,n+1}(q) - (-1)^{n+1} q^{(n+1)(m+n+3)} \\*
  &\hspace*{\parindent} \ + q^{m+3} \left( C_{m+4,n}(q) - (-1)^n q^{n(m+n+4)} \right) \\
  &= \ C_{m+2,n+1}(q) + q^{m+3} C_{m+4,n}(q). \qedhere
 \end{align*}
\end{proof}

For all integers $m,n \geqs 0$, let us set
\begin{align}\label{E:D_def}
 D_{m,n}(q) &:= \sum_{k=0}^n (-1)^k q^{-k(2k+4m+n+1)} \sqbinom{n}{k}_q C_{-4k-2m-1,m}(q^{-1}).
\end{align}

\begin{lemma}
 For all integers $m \geqs 0$ and $n > 0$, the polynomial $D_{m,n}(q) \in \Z[q,q^{-1}]$ satisfies the recurrence relation
 \begin{align}\label{E:D_rec}
  D_{m,n}(q) = \left( 1 - q^{-2(m+n)} \right) D_{m,n-1}(q) - q^{-2(2m+n+1)} D_{m+1,n-1}(q).
 \end{align}
\end{lemma}

\begin{proof}
 We have
 \begin{align*}
  D_{m,n}(q) \
  &\stackrel{\mathclap{\eqref{E:D_def}}}{=} \ \sum_{k=0}^n (-1)^k q^{-k(2k+4m+n+1)} \sqbinom{n}{k}_q C_{-4k-2m-1,m}(q^{-1}) \\
  &\stackrel{\mathclap{\eqref{E:Tartaglia_1}}}{=} \ \sum_{k=0}^n (-1)^k q^{-k(2k+4m+n+1)} \left( q^k \sqbinom{n-1}{k}_q + q^{k-n} \sqbinom{n-1}{k-1}_q \right) \\*
  &\hspace*{\parindent} \ C_{-4k-2m-1,m}(q^{-1}) \\
  &\stackrel{\mathclap{\eqref{E:C_prop}}}{=} \ \sum_{k=0}^n (-1)^k q^{-k(2k+4m+n)} \sqbinom{n-1}{k}_q C_{-4k-2m-1,m}(q^{-1}) \\*
  &\hspace*{\parindent} \ - q^{-2(2m+n+1)} \sum_{k=0}^n (-1)^{k-1} q^{-(k-1)(2k+4m+n+2)} \sqbinom{n-1}{k-1}_q \\*
  &\hspace*{\parindent} \ \left( C_{-4k-2m+1,m+1}(q^{-1}) + q^{4k+2m-2} C_{-4k-2m+3,m}(q^{-1}) \right) \\
  &\stackrel{\mathclap{\eqref{E:D_def}}}{=} \ D_{m,n-1}(q) - q^{-2(2m+n+1)} \left( D_{m+1,n-1}(q) + q^{2m+2} D_{m,n-1}(q) \right). \qedhere
 \end{align*}
\end{proof}

\begin{lemma}
 For all integers $0 \leqs m < n \leqs r-1$ we have 
 \begin{align}\label{E:D_prop}
  D_{r-m-1,n}(\zeta) = 0.
 \end{align}
\end{lemma}

\begin{proof}
 Let us prove Equation~\eqref{E:D_prop} by induction on $0 < n \leqs r-1$ for all $0 \leqs m < n$. If $n=1$ and $m=0$, then we have
 \begin{align*}
  D_{r-1,1}(\zeta) \
  &\stackrel{\mathclap{\eqref{E:D_def}}}{=} \ \sum_{k=0}^1 (-1)^k \zeta^{-2k(k-1)} \sqbinom{1}{k}_\zeta C_{-4k+1,r-1}(\zeta^{-1}) \\
  &\stackrel{\mathclap{\eqref{E:C_def}}}{=} \ \sum_{k=0}^1 (-1)^k \zeta^{-2k(k-1)} \sqbinom{1}{k}_\zeta \left( \sum_{\ell=0}^{r-1} (-1)^\ell \zeta^{-2\ell(\ell-2k+1)} \sqbinom{r-1}{\ell}_\zeta \right) \\  
  &\stackrel{\mathclap{\eqref{E:Gauss_sum}}}{=} \ G_{-1} - G_1 \\
  &= \ 0,
 \end{align*}
 where the third equality uses the identity
 \begin{align}\label{E:binomial_inversion_at_roots_of_unity}
  \sqbinom{r-m-1}{\ell}_\zeta = (-1)^\ell \sqbinom{\ell+m}{\ell}_\zeta,
 \end{align}
 which holds for all integers $0 \leqs m \leqs r-1$ and $0 \leqs \ell \leqs r-m-1$. If $1 < n \leqs r-1$ and $0 \leqs m < n$, then we have
 \begin{align*}
  D_{r-m-1,n}(\zeta) \
  &\stackrel{\mathclap{\eqref{E:D_rec}}}{=} \ \left( 1 - \zeta^{2(m-n+1)} \right) D_{r-m-1,n-1}(\zeta) - \zeta^{2(2m-n+1)} D_{r-m,n-1}(\zeta) \\
  &= \ 0.
 \end{align*}
 Notice that, when $m = n-1$, the induction hypothesis does not allow us to recover the value of $D_{r-m-1,n-1}(\zeta)$, but in this case $1 - \zeta^{2(m-n+1)} = 0$.
\end{proof}

For all $m_1,m_2,n_1,n_2 \in \Z$ with $n_1,n_2 \geqs 0$, let us set
\begin{align}
 E_{m_1,m_2,n_1,n_2}(q) &:= \sum_{k_1=0}^{n_1} \sum_{k_2=0}^{n_2} (-1)^{k_1+k_2} q^{2(k_1-k_2)^2+k_1(m_1-n_1)+k_2(m_2-n_2)} \nonumber \\*
 &\hspace*{\parindent} \sqbinom{n_1}{k_1}_q \sqbinom{n_2}{k_2}_q. \label{E:E_def}
\end{align}

\section{Computations}\label{A:computations}

In this section, we prove results that were announced in Section~\ref{S:triangular_basis}.

\begin{proof}[Proof of Lemma~\ref{L:v_to_t}]
 It is enough to prove that 
 \begin{align*}
  \sum_{k=0}^n (-1)^{n-k} \zeta^{(n-k)(n-1)} \sqbinom{n}{k}_\zeta \left( \sum_{m=0}^k \zeta^{m(k-m)} \sqbinom{k}{m}_\zeta t_m \right) 
  &= t_n.
 \end{align*}  
 This follows from the computation
 \begin{align*}
  &\sum_{k=0}^n (-1)^{n-k} \zeta^{(n-k)(n-1)} \sqbinom{n}{k}_\zeta \left( \sum_{m=0}^k \zeta^{m(k-m)} \sqbinom{k}{m}_\zeta t_m \right) \\
  &\hspace*{\parindent} = \ (-1)^n \zeta^{n(n-1)} \sum_{m=0}^n \zeta^{-m^2} \left( \sum_{k=m}^n (-1)^k \zeta^{k(m-n+1)} \sqbinom{n}{k}_\zeta \sqbinom{k}{m}_\zeta \right) t_m \\
  &\hspace*{\parindent} \stackrel{\mathclap{\eqref{E:A_def}}}{=} \ (-1)^n \zeta^{n(n-1)} \sum_{m=0}^n \zeta^{-m^2} A_{0,m,n}(\zeta) t_m \\
  &\hspace*{\parindent} \stackrel{\mathclap{\eqref{E:A_prop}}}{=} \ \zeta^{n^2} \sum_{m=0}^n \zeta^{-m^2} \sqbinom{0}{m-n}_\zeta t_m \
  = \ t_n. \qedhere
 \end{align*}
\end{proof}

\begin{proof}[Proof of Proposition~\ref{P:triangularity}]
 For $\alpha$, we have
 \begin{align*}
  \alpha \cdot t_n \ 
  &\stackrel{\mathclap{\eqref{E:t_to_v}}}{=} \ \sum_{k=0}^n (-1)^{n-k} \zeta^{(n-k)(n-1)} \sqbinom{n}{k}_\zeta \alpha \cdot v_k \\
  &\stackrel{\mathclap{\eqref{E:alpha}}}{=} \ \sum_{k=0}^n (-1)^{n-k} \zeta^{(n-k)(n-1)+4k} \sqbinom{n}{k}_\zeta v_k \\
  &\stackrel{\mathclap{\eqref{E:v_to_t}}}{=} \ \sum_{k=0}^n (-1)^{n-k} \zeta^{(n-k)(n-1)+4k} \sqbinom{n}{k}_\zeta \left( \sum_{m=0}^k \zeta^{m(k-m)} \sqbinom{k}{m}_\zeta t_m \right) \\
  &= \ (-1)^n \zeta^{n(n-1)} \sum_{m=0}^n \zeta^{-m^2} \\*
  &\hspace*{\parindent} \ \left( \sum_{k=m}^n (-1)^k \zeta^{k(m-n+5)} \sqbinom{n}{k}_\zeta \sqbinom{k}{m}_\zeta \right) t_m \\
  &= \ (-1)^n \zeta^{n(n-1)} \sum_{m=0}^n \zeta^{-m^2} \\*
  &\hspace*{\parindent} \ \left( \sum_{k=0}^{n-m} (-1)^{k+m} \zeta^{(k+m)(m-n+5)} \sqbinom{n}{k+m}_\zeta \sqbinom{k+m}{m}_\zeta \right) t_m \\
  &= \ (-1)^n \zeta^{n(n-1)} \sum_{m=0}^n (-1)^m \zeta^{-m(n-5)} \sqbinom{n}{m}_\zeta \\*
  &\hspace*{\parindent} \ \left( \sum_{k=0}^{n-m} (-1)^k \zeta^{k(m-n+5)} \sqbinom{n-m}{k}_\zeta \right) t_m \\
  &\stackrel{\mathclap{\eqref{E:B_def}}}{=} \ (-1)^n \zeta^{n(n-1)} \sum_{m=0}^n (-1)^m \zeta^{-m(n-5)} \sqbinom{n}{m}_\zeta B_{n-m}(\zeta) t_m \\
  &\stackrel{\mathclap{\eqref{E:B_prop}}}{=} \ \zeta^{n(n-1)} \sum_{m=0}^n \zeta^{-m(n-5)-\frac{(n-m)(n-m-5)}{2}} \{ 2;n-m \}_\zeta \sqbinom{n}{m}_\zeta t_m \\
  &= \ \zeta^{4n} \sum_{k=0}^n \zeta^{-\frac{k(k-2n+5)}{2}} \{ 2;k \}_\zeta \sqbinom{n}{k}_\zeta t_{n-k}.
 \end{align*}
 Notice that, when $k > 2$, we have $\{ 2;k \}_\zeta = 0$. For $\beta$, we have
 \begin{align*}
  \beta \cdot t_n \
  &\stackrel{\mathclap{\eqref{E:t_to_v}}}{=} \ \sum_{k=0}^n (-1)^{n-k} \zeta^{(n-k)(n-1)} \sqbinom{n}{k}_\zeta \beta \cdot v_k \\
  &\stackrel{\mathclap{\eqref{E:beta}}}{=} \ \sum_{k=0}^n (-1)^{n-k} \zeta^{(n-k)(n-1)} \sqbinom{n}{k}_\zeta v_{k+1} \\
  &\stackrel{\mathclap{\eqref{E:v_to_t}}}{=} \ \sum_{k=0}^n (-1)^{n-k} \zeta^{(n-k)(n-1)} \sqbinom{n}{k}_\zeta \left( \sum_{m=0}^{k+1} \zeta^{m(k-m+1)} \sqbinom{k+1}{m}_\zeta t_m \right) \\
  &= \ (-1)^n \zeta^{n(n-1)} \sum_{m=0}^{n+1} \zeta^{-m(m-1)} \left( \sum_{k=m-1}^n (-1)^k \zeta^{k(m-n+1)} \sqbinom{n}{k}_\zeta \sqbinom{k+1}{m}_\zeta \right) t_m \\
  &\stackrel{\mathclap{\eqref{E:A_def}}}{=} \ (-1)^n \zeta^{n(n-1)} \sum_{m=0}^{n+1} \zeta^{-m(m-1)} A_{1,m,n}(\zeta) t_m \\
  &\stackrel{\mathclap{\eqref{E:A_prop}}}{=} \ \zeta^{(n+1)n} \sum_{m=0}^{n+1} \zeta^{-m(m-1)} \sqbinom{1}{m-n}_\zeta t_m \\
  &= \ \zeta^{2n} \sum_{k=-n}^1 \zeta^{-k(k+2n-1)} \sqbinom{1}{k}_\zeta t_{n+k}.
 \end{align*}
 Notice that, when $k < 0$, we have $\sqbinom{1}{k}_\zeta = 0$. For $\tau_\alpha$, we have
 \begin{align*}
  \tau_\alpha \cdot t_n \
  &\stackrel{\mathclap{\eqref{E:t_to_v}}}{=} \ \sum_{k=0}^n (-1)^{n-k} \zeta^{(n-k)(n-1)} \sqbinom{n}{k}_\zeta \tau_\alpha \cdot v_k \\
  &\stackrel{\mathclap{\eqref{E:tau_alpha_psi}}}{\propto} \ \sum_{k=0}^n (-1)^{n-k} \zeta^{(n-k)(n-1)+2(k+1)k} \sqbinom{n}{k}_\zeta v_k \\
  &\stackrel{\mathclap{\eqref{E:v_to_t}}}{=} \ \sum_{k=0}^n (-1)^{n-k} \zeta^{(n-k)(n-1)+2(k+1)k} \sqbinom{n}{k}_\zeta \left( \sum_{m=0}^k \zeta^{m(k-m)} \sqbinom{k}{m}_\zeta t_m \right) \\
  &= \ (-1)^n \zeta^{n(n-1)} \sum_{m=0}^n \zeta^{-m^2} \\*
  &\hspace*{\parindent} \ \left( \sum_{k=m}^n (-1)^k \zeta^{k(2k+m-n+3)} \sqbinom{n}{k}_\zeta \sqbinom{k}{m}_\zeta \right) t_m \\
  &= \ (-1)^n \zeta^{n(n-1)} \sum_{m=0}^n \zeta^{-m^2} \\*
  &\hspace*{\parindent} \ \left( \sum_{k=0}^{n-m} (-1)^{k+m} \zeta^{(k+m)(2k+3m-n+3)} \sqbinom{n}{k+m}_\zeta \sqbinom{k+m}{m}_\zeta \right) t_m \\
  &= \ (-1)^n \zeta^{n(n-1)} \sum_{m=0}^n (-1)^m \zeta^{m(2m-n+3)} \sqbinom{n}{m}_\zeta \\*
  &\hspace*{\parindent} \ \left( \sum_{k=0}^{n-m} (-1)^k \zeta^{k(2k+5m-n+3)} \sqbinom{n-m}{k}_\zeta \right) t_m \\
  &\stackrel{\mathclap{\eqref{E:C_def}}}{=} \ (-1)^n \zeta^{n(n-1)} \sum_{m=0}^n (-1)^m \zeta^{m(2m-n+3)} \sqbinom{n}{m}_\zeta C_{4m+3,n-m}(\zeta) t_m \\
  &= \ \zeta^{2(n+1)n} \sum_{k=0}^n (-1)^k \zeta^{k(2k-3n-3)} \sqbinom{n}{k}_\zeta C_{4(n-k)+3,k}(\zeta) t_{n-k}.
 \end{align*}
 For $\tau_\beta$, we have
 \begin{align*}
  \tau_\beta \cdot t_n \
  &\stackrel{\mathclap{\eqref{E:t_to_v}}}{=} \ \sum_{k=0}^n (-1)^{n-k} \zeta^{(n-k)(n-1)} \sqbinom{n}{k}_\zeta \tau_\beta \cdot v_k \\
  &\stackrel{\mathclap{\eqref{E:tau_beta_psi}}}{\propto} \ \frac{1}{G_1} \sum_{k=0}^n \sum_{\ell=0}^{r-1} (-1)^{n-k} \zeta^{(n-k)(n-1)-2(k-\ell)(k-\ell-1)} \sqbinom{n}{k}_\zeta v_\ell \\
  &\stackrel{\mathclap{\eqref{E:v_to_t}}}{=} \ \frac{1}{G_1} \sum_{k=0}^n \sum_{\ell=0}^{r-1} (-1)^{n-k} \zeta^{(n-k)(n-1)-2(k-\ell)(k-\ell-1)} \sqbinom{n}{k}_\zeta \\*
  &\hspace*{\parindent} \ \left( \sum_{m=0}^\ell \zeta^{m(\ell-m)} \sqbinom{\ell}{m}_\zeta t_m \right) \\
  &= \ (-1)^n \frac{\zeta^{n(n-1)}}{G_1} \sum_{m=0}^{r-1} \zeta^{-m^2} \left( \sum_{k=0}^n (-1)^k \zeta^{-k(2k+n-3)} \sqbinom{n}{k}_\zeta \right. \\*
  &\hspace*{\parindent} \ \left. \left( \sum_{\ell=m}^{r-1} \zeta^{-\ell(2\ell-4k-m+2)} \sqbinom{\ell}{m}_\zeta \right) \right) t_m \\
  &= \ (-1)^n \frac{\zeta^{n(n-1)}}{G_1} \sum_{m=0}^{r-1} \zeta^{-m^2} \left( \sum_{k=0}^n (-1)^k \zeta^{-k(2k+n-3)} \sqbinom{n}{k}_\zeta \right. \\*
  &\hspace*{\parindent} \ \left. \left( \sum_{\ell=0}^{r-m-1} \zeta^{-(\ell+m)(2\ell-4k+m+2)} \sqbinom{\ell+m}{m}_\zeta \right) \right) t_m \\
  &\stackrel{\mathclap{\eqref{E:binomial_inversion_at_roots_of_unity}}}{=} \ (-1)^n \frac{\zeta^{n(n-1)}}{G_1} \sum_{m=0}^{r-1} \zeta^{-2(m+1)m} \left( \sum_{k=0}^n (-1)^k \zeta^{-k(2k-4m+n-3)} \sqbinom{n}{k}_\zeta \right. \\*
  &\hspace*{\parindent} \ \left. \left( \sum_{\ell=0}^{r-m-1} (-1)^\ell \zeta^{-\ell(2\ell-4k+3m+2)} \sqbinom{r-m-1}{\ell}_\zeta \right) \right) t_m \\
  &\stackrel{\mathclap{\eqref{E:D_def}}}{=} \ (-1)^n \frac{\zeta^{n(n-1)}}{G_1} \sum_{m=0}^{r-1} \zeta^{-2(m+1)m} D_{r-m-1,n}(\zeta) t_m \\
  &= \ (-1)^n \frac{\zeta^{-(n+3)n}}{G_1} \sum_{k=-n}^{r-n-1} \zeta^{-2k(k+2n+1)} D_{r-n-k-1,n}(\zeta) t_{n+k}.
 \end{align*}
 Notice that, when $k < 0$, we have $D_{r-n-k-1,n}(\zeta) = 0$, thanks to Equation~\eqref{E:D_prop}. For $\tau_\gamma$, we have
 \begin{align*}
  \tau_\gamma \cdot (t_{n_1} \otimes t_{n_2}) \
  &\stackrel{\mathclap{\eqref{E:t_to_v}}}{=} \ \sum_{k_1=0}^{n_1} \sum_{k_2=0}^{n_2} (-1)^{n_1-k_1+n_2-k_2} \zeta^{(n_1-k_1)(n_1-1)+(n_2-k_2)(n_2-1)} \\*
  &\hspace*{\parindent} \ \sqbinom{n_1}{k_1}_\zeta \sqbinom{n_2}{k_2}_\zeta \tau_\gamma \cdot (v_{k_1} \otimes v_{k_2}) \\
  &\stackrel{\mathclap{\eqref{E:tau_gamma_psi}}}{\propto} \ \sum_{k_1=0}^{n_1} \sum_{k_2=0}^{n_2} (-1)^{n_1-k_1+n_2-k_2} \zeta^{(n_1-k_1)(n_1-1)+(n_2-k_2)(n_2-1)} \\*
  &\hspace*{\parindent} \ \zeta^{2(k_1-k_2+1)(k_1-k_2)} \sqbinom{n_1}{k_1}_\zeta \sqbinom{n_2}{k_2}_\zeta v_{k_1} \otimes v_{k_2} \\
  &\stackrel{\mathclap{\eqref{E:v_to_t}}}{=} \ \sum_{k_1=0}^{n_1} \sum_{k_2=0}^{n_2} (-1)^{n_1-k_1+n_2-k_2} \zeta^{(n_1-k_1)(n_1-1)+(n_2-k_2)(n_2-1)} \\*
  &\hspace*{\parindent} \ \zeta^{2(k_1-k_2+1)(k_1-k_2)} \sqbinom{n_1}{k_1}_\zeta \sqbinom{n_2}{k_2}_\zeta \\*
  &\hspace*{\parindent} \ \left( \sum_{m_1=0}^{k_1} \sum_{m_2=0}^{k_2} \zeta^{m_1(k_1-m_1)+m_2(k_2-m_2)} \sqbinom{k_1}{m_1}_\zeta \sqbinom{k_2}{m_2}_\zeta t_{m_1} \otimes t_{m_2} \right) \\
  &= \ (-1)^{n_1+n_2} \zeta^{n_1(n_1-1)+n_2(n_2-1)} \sum_{m_1=0}^{n_1} \sum_{m_2=0}^{n_2} \zeta^{-m_1^2-m_2^2} \\*
  &\hspace*{\parindent} \ \left( \sum_{k_1=m_1}^{n_1} \sum_{k_2=m_2}^{n_2} (-1)^{k_1+k_2} \zeta^{2(k_1-k_2)^2+k_1(m_1-n_1+3)+k_2(m_2-n_2-1)} \right. \\*
  &\hspace*{\parindent} \ \left. \sqbinom{n_1}{k_1}_\zeta \sqbinom{k_1}{m_1}_\zeta \sqbinom{n_2}{k_2}_\zeta \sqbinom{k_2}{m_2}_\zeta \right) t_{m_1} \otimes t_{m_2} \\
  &= \ (-1)^{n_1+n_2} \zeta^{n_1(n_1-1)+n_2(n_2-1)} \sum_{m_1=0}^{n_1} \sum_{m_2=0}^{n_2} \zeta^{-m_1^2-m_2^2} \\*
  &\hspace*{\parindent} \ \left( \sum_{k_1=0}^{n_1-m_1} \sum_{k_2=0}^{n_2-m_2} (-1)^{k_1+m_1+k_2+m_2} \zeta^{2(k_1+m_1-k_2-m_2)^2} \right. \\*
  &\hspace*{\parindent} \ \zeta^{(k_1+m_1)(m_1-n_1+3)+(k_2+m_2)(m_2-n_2-1)} \\*
  &\hspace*{\parindent} \ \left. \sqbinom{n_1}{k_1+m_1}_\zeta \sqbinom{k_1+m_1}{m_1}_\zeta \sqbinom{n_2}{k_2+m_2}_\zeta \sqbinom{k_2+m_2}{m_2}_\zeta \right) t_{m_1} \otimes t_{m_2} \\  
  &= \ (-1)^{n_1+n_2} \zeta^{n_1(n_1-1)+n_2(n_2-1)} \sum_{m_1=0}^{n_1} \sum_{m_2=0}^{n_2} (-1)^{m_1+m_2} \\*
  &\hspace*{\parindent} \ \zeta^{2(m_1-m_2)^2-m_1(n_1-3)-m_2(n_2+1)} \sqbinom{n_1}{m_1}_\zeta \sqbinom{n_2}{m_2}_\zeta \\*
  &\hspace*{\parindent} \ \left( \sum_{k_1=0}^{n_1-m_1} \sum_{k_2=0}^{n_2-m_2} (-1)^{k_1+k_2} \zeta^{2(k_1-k_2)(k_1-k_2+2(m_1-m_2)+1)} \right. \\*
  &\hspace*{\parindent} \ \left. \zeta^{k_1(m_1-n_1+1)+k_2(m_2-n_2+1)}
  \sqbinom{n_1-m_1}{k_1}_\zeta \sqbinom{n_2-m_2}{k_2}_\zeta \right) t_{m_1} \otimes t_{m_2} \\  
  &\stackrel{\mathclap{\eqref{E:E_def}}}{=} \ (-1)^{n_1+n_2} \zeta^{n_1(n_1-1)+n_2(n_2-1)} \sum_{m_1=0}^{n_1} \sum_{m_2=0}^{n_2} (-1)^{m_1+m_2} \\*
  &\hspace*{\parindent} \ \zeta^{2(m_1-m_2)^2-m_1(n_1-3)-m_2(n_2+1)} \sqbinom{n_1}{m_1}_\zeta \sqbinom{n_2}{m_2}_\zeta \\*
  &\hspace*{\parindent} \ E_{4(m_1-m_2)+3,-4(m_1-m_2)-1,n_1-m_1,n_2-m_2}(\zeta) t_{m_1} \otimes t_{m_2} \\  
  &= \ \zeta^{2(n_1-n_2+1)(n_1-n_2)} \sum_{k_1=0}^{n_1} \sum_{k_2=0}^{n_2} (-1)^{k_1+k_2} \\*
  &\hspace*{\parindent} \ \zeta^{2(k_1-k_2)^2-k_1(3n_1-4n_2+3)+k_2(4n_1-3n_2+1)} \sqbinom{n_1}{k_1}_\zeta \sqbinom{n_2}{k_2}_\zeta \\*
  &\hspace*{\parindent} \ E_{4(n_1-k_1-n_2+k_2)+3,-4(n_1-k_1-n_2+k_2)-1,k_1,k_2}(\zeta) t_{n_1-k_1} \otimes t_{n_2-k_2}.
 \end{align*}
\end{proof}

\section{Divisibility results}\label{A:divisibility}

For all $\ell,m \in \Z$, let us set
\begin{align}\label{E:c_def}
 c(\ell,m) &:= \ell(2\ell+m).
\end{align}

\begin{remark}
 Notice that we have
 \begin{align}\label{E:c_rec}
  c(\ell,m) &= c(\ell-1,m+4)+m+2
 \end{align}
 for all $\ell,m \in \Z$. Furthermore, we have
 \begin{align}\label{E:c_binom}
  c(k,m-n) \binom{n}{k} 
  &= n \left( -2(n-1) \binom{n-2}{k-1} + (m+n) \binom{n-1}{k-1} \right).
 \end{align}
 for all integers $0 \leqs k \leqs n$ and $m \in \Z$.
\end{remark}
 
For all integers $\ell,n \geqs 0$ and $m \in \Z$, let us set
\begin{align}\label{E:P_def}
 P_{\ell,m,n}(q) &:= \sum_{k=0}^n (-1)^k c(k,m-n)^\ell \binom{n}{k} q^{c(k,m-n)}.
\end{align}

\begin{remark}
 For all integers $\ell,n \geqs 0$ and $m \in \Z$, we have
 \begin{align}\label{E:P_deriv}
  \left( \frac{\rmd}{\rmd q} P_{\ell,m,n} \right) (q) &= q^{-1} P_{\ell+1,m,n}(q).
 \end{align}
\end{remark}

\begin{lemma}
 For all integers $\ell,n \geqs 0$ and $m \in \Z$, we have
 \begin{align}\label{E:P_prop}
  P_{\ell,m,n}(q)
  &= n q^{m-n+2} \sum_{j=0}^{\ell-1} \binom{\ell-1}{j} (m-n+2)^{\ell-j-1} \nonumber \\*
  &\hspace*{\parindent} \bigg( 2(n-1) P_{j,m+2,n-2}(q) - (m+n) P_{j,m+3,n-1}(q) \bigg).
 \end{align}
\end{lemma}

\begin{proof}
 We have
 \begin{align*}
  P_{\ell,m,n}(q) \
  &\stackrel{\mathclap{\eqref{E:P_def}}}{=} \ \sum_{k=0}^n (-1)^k c(k,m-n)^\ell \binom{n}{k} q^{c(k,m-n)} \\
  &\stackrel{\mathclap{\eqref{E:c_binom}}}{=} \ \sum_{k=0}^n (-1)^k c(k,m-n)^{\ell-1}  \\*
  &\hspace*{\parindent} n \left( -2(n-1) \binom{n-2}{k-1} + (m+n) \binom{n-1}{k-1} \right) q^{c(k,m-n)} \\
  &\stackrel{\mathclap{\eqref{E:c_rec}}}{=} \ \sum_{k=0}^n (-1)^k \left( c(k-1,m-n+4)+m-n+2 \right)^{\ell-1} \\*
  &\hspace*{\parindent} n q^{m-n+2} \left( -2(n-1) \binom{n-2}{k-1} + (m+n) \binom{n-1}{k-1} \right) q^{c(k-1,m-n+4)} \\
  &\stackrel{\mathclap{\eqref{E:P_def}}}{=} \ \sum_{j=0}^{\ell-1} \binom{\ell-1}{j} (m-n+2)^{\ell-j-1} \\*
  &\hspace*{\parindent} n q^{m-n+2} \bigg( -2(n-1) P_{j,m+2,n-2} + (m+n) P_{j,m+3,n-1} \bigg). \qedhere
 \end{align*}
\end{proof}

\begin{lemma}
 For all integers $m \in \Z$ and $n \geqs 0$, we have
 \begin{align}\label{E:P_div}
  h(q)^{\lfloor \frac{n+1}{2} \rfloor} \mid P_{0,m,n}(q),
 \end{align}
 where $h(q) := 1-q \in \Z[q,q^{-1}]$.
\end{lemma}

\begin{proof}
 Let us prove that
 \begin{align*}
  \left( \left( \frac{\rmd}{\rmd q} \right)^\ell P_{0,m,n} \right)(1) = 0
 \end{align*}
 for every integer $0 \leqs \ell \leqs \frac{n-1}{2}$. Thanks to Equation~\ref{E:P_deriv}, we can equivalently check that 
 \begin{align*}
  P_{\ell,m,n}(1) = 0
 \end{align*}
 for every integer $0 \leqs \ell \leqs \frac{n-1}{2}$. Equation~\eqref{E:P_prop} implies that
 \begin{align*}
  P_{\ell,m,n}(q) \in \left\langle P_{0,i,j}(q) \biggm| 
  \begin{array}{l}
   m \leqs i \leqs m+3\ell \\
   n-2\ell \leqs j \leqs n
  \end{array} \right\rangle_\Z.
 \end{align*}
 The claim then follows from 
 \begin{align*}
  P_{0,i,j}(1) \
  &\stackrel{\mathclap{\eqref{E:P_def}}}{=} \ \sum_{k=0}^j (-1)^k \binom{j}{k} \
  = \ \delta_{j,0}. \qedhere
 \end{align*}
\end{proof}

For all $\ell_1,\ell_2,m_1,m_2 \in \Z$, let us set
\begin{align}\label{E:e_def}
 e(\ell_1,\ell_2,m_1,m_2) &:= c(\ell_1,m_1) + c(\ell_2,m_2) - 4\ell_1\ell_2.
\end{align}

\begin{remark}
 Notice that we have
 \begin{align}
  e(\ell_1,\ell_2,m_1,m_2)
  &= e(\ell_1-1,\ell_2,m_1+4,m_2-4)+m_1+2 \nonumber \\*
  &= e(\ell_1,\ell_2-1,m_1-4,m_2+4)+m_2+2 \nonumber \\*
  &= e(\ell_1-1,\ell_2-1,m_1,m_2)+m_1+m_2 \label{E:e_rec}
 \end{align}
 for all $\ell_1,\ell_2,m_1,m_2 \in \Z$. Furthermore, we have
 \begin{align}
  &f(k_1,k_2,m_1-n_1,m_2-n_2) \binom{n_1}{k_1} \binom{n_2}{k_2} \nonumber \\*
  &\hspace*{\parindent} = n_1 \left( -2(n_1-1) \binom{n_1-2}{k_1-1} + (m_1+n_1) \binom{n_1-1}{k_1-1} \right) \binom{n_2}{k_2} \nonumber \\*
  &\hspace*{\parindent} {} + n_2 \left( -2(n_2-1) \binom{n_2-2}{k_2-1} + (m_2+n_2) \binom{n_2-1}{k_2-1} \right) \binom{n_1}{k_1} \nonumber \\*
  &\hspace*{\parindent} {} - 2 n_1 n_2 \binom{n_1-1}{k_1-1} \binom{n_2-1}{k_2-1} \label{E:e_binom}
 \end{align}
 for all integers $0 \leqs k_1 \leqs n_1$, $0 \leqs k_2 \leqs n_2$, and $m_1,m_2 \in \Z$.
\end{remark}

For all integers $\ell,n_1,n_2 \geqs 0$ and $m_1,m_2 \in \Z$, let us set
\begin{align}
 Q_{\ell,m_1,m_2,n_1,n_2}(q) &:= \sum_{k_1=0}^{n_1} \sum_{k_2=0}^{n_2} (-1)^{k_1+k_2} f(k_1,k_2,m_1-n_1,m_2-n_2)^\ell \binom{n_1}{k_1} \binom{n_2}{k_2} \nonumber \\*
 &\hspace*{\parindent} q^{f(k_1,k_2,m_1-n_1,m_2-n_2)}. \label{E:Q_def}
\end{align}

\begin{remark}
 For all integers $\ell,n_1,n_2 \geqs 0$ and $m_1,m_2 \in \Z$, we have
 \begin{align}\label{E:Q_deriv}
  \left( \frac{\rmd}{\rmd q} Q_{\ell,m_1,m_2,n_1,n_2} \right) (q) &= q^{-1} Q_{\ell+1,m_1,m_2,n_1,n_2}(q).
 \end{align}
\end{remark}

\begin{lemma}
 For all integers $\ell,n_1,n_2 \geqs 0$ and $m_1,m_2 \in \Z$, we have
 \begin{align}
  &Q_{\ell,m_1,m_2,n_1,n_2}(q) \nonumber \\*
  &\hspace*{\parindent} = \sum_{j=0}^{\ell-1} \binom{\ell-1}{j} \bigg( (m_1-n_1+2)^{\ell-j-1} n_1 q^{m_1-n_1+2} \nonumber \\*
  &\hspace*{2\parindent} \Big( -2(n_1-1) Q_{j,m_1+2,m_2-4,n_1-2,n_2}(q) 
  + (m_1+n_1) Q_{j,m_1+3,m_2-4,n_1-1,n_2}(q) \Big) \nonumber \\*
  &\hspace*{2\parindent} + (m_2-n_2+2)^{\ell-j-1} n_2 q^{m_2-n_2+2} \nonumber \\*
  &\hspace*{2\parindent} \Big( -2(n_2-1) Q_{j,m_1-4,m_2+2,n_1,n_2-2}(q) 
  + (m_2+n_2) Q_{j,m_1-4,m_2+3,n_1,n_2-1}(q) \Big) \nonumber \\*
  &\hspace*{2\parindent} - (m_1+m_2-n_1-n_2)^{\ell-j-1} 2 n_1 n_2 q^{m_1+m_2-n_1-n_2} \nonumber \\* 
  &\hspace*{2\parindent} Q_{j,m_1-1,m_2-1,n_1-1,n_2-1}(q) \bigg). \label{E:Q_prop}
 \end{align}
\end{lemma}

\begin{proof}
 We have
 \begin{align*}
  &Q_{\ell,m_1,m_2,n_1,n_2}(q) \\
  &\hspace*{\parindent} \stackrel{\mathclap{\eqref{E:Q_def}}}{=} \ \sum_{k_1=0}^{n_1} \sum_{k_2=0}^{n_2} (-1)^{k_1+k_2} f(k_1,k_2,m_1-n_1,m_2-n_2)^\ell \binom{n_1}{k_1} \binom{n_2}{k_2} \\*
  &\hspace*{2\parindent} q^{f(k_1,k_2,m_1-n_1,m_2-n_2)} \\
  &\hspace*{\parindent} \stackrel{\mathclap{\eqref{E:e_binom}}}{=} \ \sum_{k_1=0}^{n_1} \sum_{k_2=0}^{n_2} (-1)^{k_1+k_2} f(k_1,k_2,m_1-n_1,m_2-n_2)^{\ell-1} \\*
  &\hspace*{2\parindent} \left( n_1 \left( -2(n_1-1) \binom{n_1-2}{k_1-1} + (m_1+n_1) \binom{n_1-1}{k_1-1} \right) \binom{n_2}{k_2} \right. \\*
  &\hspace*{2\parindent} {} + n_2 \left( -2(n_2-1) \binom{n_2-2}{k_2-1} + (m_2+n_2) \binom{n_2-1}{k_2-1} \right) \binom{n_1}{k_1} \\*
  &\hspace*{2\parindent} {} \left. - 2 n_1 n_2 \binom{n_1-1}{k_1-1} \binom{n_2-1}{k_2-1} \right)
  q^{f(k_1,k_2,m_1-n_1,m_2-n_2)} \\
  &\hspace*{\parindent} \stackrel{\mathclap{\eqref{E:e_rec}}}{=} \ \sum_{k_1=0}^{n_1} \sum_{k_2=0}^{n_2} (-1)^{k_1+k_2} \\*
  &\hspace*{2\parindent} \bigg( f(k_1-1,k_2,m_1-n_1+4,m_2-n_2-4)+m_1-n_1+2 \bigg)^{\ell-1} \\*
  &\hspace*{2\parindent} n_1 q^{m_1-n_1+2} \left( -2(n_1-1) \binom{n_1-2}{k_1-1} + (m_1+n_1) \binom{n_1-1}{k_1-1} \right) \\*
  &\hspace*{2\parindent} \binom{n_2}{k_2} q^{f(k_1-1,k_2,m_1-n_1+4,m_2-n_2-4)} \\*
  &\hspace*{2\parindent} + \bigg( f(k_1,k_2-1,m_1-n_1-4,m_2-n_2+4)+m_2-n_2+2 \bigg)^{\ell-1} \\*
  &\hspace*{2\parindent} + n_2 q^{m_2-n_2+2} \left( -2(n_2-1) \binom{n_2-2}{k_2-1} + (m_2+n_2) \binom{n_2-1}{k_2-1} \right) \\*
  &\hspace*{2\parindent} \binom{n_1}{k_1} q^{f(k_1,k_2-1,m_1-n_1-4,m_2-n_2+4)} \\*
  &\hspace*{2\parindent} - \bigg( f(k_1-1,k_2-1,m_1-n_1,m_2-n_2)+m_1+m_2-n_1-n_2 \bigg)^{\ell-1} \\* 
  &\hspace*{2\parindent} 2 n_1 n_2 q^{m_1+m_2-n_1-n_2} \binom{n_1-1}{k_1-1} \binom{n_2-1}{k_2-1} 
  q^{f(k_1-1,k_2-1,m_1-n_1,m_2-n_2)} \\
  &\hspace*{\parindent} \stackrel{\mathclap{\eqref{E:Q_def}}}{=} \ \sum_{j=0}^{\ell-1} \binom{\ell-1}{j} \bigg( (m_1-n_1+2)^{\ell-j-1} n_1 q^{m_1-n_1+2} \\*
  &\hspace*{2\parindent} \Big( -2(n_1-1) Q_{j,m_1+2,m_2-4,n_1-2,n_2}(q) 
  + (m_1+n_1) Q_{j,m_1+3,m_2-4,n_1-1,n_2}(q) \Big) \\*
  &\hspace*{2\parindent} + (m_2-n_2+2)^{\ell-j-1} n_2 q^{m_2-n_2+2} \\*
  &\hspace*{2\parindent} \Big( -2(n_2-1) Q_{j,m_1-4,m_2+2,n_1,n_2-2}(q) 
  + (m_2+n_2) Q_{j,m_1-4,m_2+3,n_1,n_2-1}(q) \Big) \\*
  &\hspace*{2\parindent} - (m_1+m_2-n_1-n_2)^{\ell-j-1} 2 n_1 n_2 q^{m_1+m_2-n_1-n_2} \\* 
  &\hspace*{2\parindent} Q_{j,m_1-1,m_2-1,n_1-1,n_2-1}(q) \bigg). \qedhere
\end{align*}
\end{proof}

\begin{lemma}
 For all integers $m_1,m_2 \in \Z$ and $n_1,n_2 \geqs 0$, we have
 \begin{align}\label{E:Q_div}
  h(q)^{\lfloor \frac{n_1+n_2+1}{2} \rfloor} \mid Q_{0,m_1,m_2,n_1,n_2}(q),
 \end{align}
 where $h(q) := 1-q \in \Z[q,q^{-1}]$.
\end{lemma}

\begin{proof}
 Let us prove that
 \begin{align*}
  \left( \left( \frac{\rmd}{\rmd q} \right)^\ell Q_{0,m_1,m_2,n_1,n_2} \right)(1) = 0
 \end{align*}
 for every integer $0 \leqs \ell \leqs \frac{n_1+n_2-1}{2}$. Thanks to Equation~\ref{E:Q_deriv}, we can equivalently check that 
 \begin{align*}
  Q_{\ell,m_1,m_2,n_1,n_2}(1) = 0
 \end{align*}
 for every integer $0 \leqs \ell \leqs \frac{n_1+n_2-1}{2}$. Equation~\eqref{E:Q_prop} implies that
 \begin{align*}
  Q_{\ell,m_1,m_2,n_1,n_2}(q) \in \left\langle Q_{0,i_1,i_2,j_1,j_2}(q) \biggm| 
  \begin{array}{l}
   m_1+m_2-2\ell \leqs i_1+i_2 \leqs m_1+m_2 \\
   n_1+n_2-2\ell \leqs j_1+j_2 \leqs n_1+n_2
  \end{array}
  \right\rangle_\Z.
 \end{align*}
 The claim then follows from 
 \begin{align*}
  Q_{0,i_1,i_2,j_1,j_2}(1) \
  &\stackrel{\mathclap{\eqref{E:Q_def}}}{=} \ \sum_{k_1=0}^{j_1} \sum_{k_2=0}^{j_2} (-1)^{k_1+k_2} \binom{j_1}{k_1} \binom{j_2}{k_2} \
  = \ \delta_{j_1,0} \delta_{j_2,0}. \qedhere
 \end{align*}
\end{proof}

We are now ready to prove Proposition~\ref{P:divisibility}.

\begin{proof}[Proof of Proposition~\ref{P:divisibility}]
 First, in order to prove Equation~\eqref{E:C_div}, we simply notice that
 \begin{align*}
  \left( \left( \frac{\rmd}{\rmd q} \right)^\ell C_{m,n} \right)(1) = 
  \left( \left( \frac{\rmd}{\rmd q} \right)^\ell P_{0,m,n} \right)(1) = 0
 \end{align*}
 for every integer $0 \leqs \ell \leqs \frac{n-1}{2}$, where the polynomial 
 \begin{align*}
  P_{\ell,m,n}(q) \in \Z[q,q^{-1}]
 \end{align*}
 is defined by Equations~\eqref{E:c_def} and \eqref{E:P_def} for all integers $\ell,n \geqs 0$ and $m \in \Z$. Indeed, if we set
 \begin{align*}
  \sqbinom{n}{k}(q) := \sqbinom{n}{k}_q,
 \end{align*}
 then we have
 \begin{align*}
  \sqbinom{n}{k}(1) &= \binom{n}{k}, &
  \left( \frac{\rmd}{\rmd q} \sqbinom{n}{k} \right)(1) &= 0.
 \end{align*}
 Then, the claim follows from Equation~\eqref{E:P_div}.
 
 Next, let us prove Equation~\eqref{E:D_div} by induction on $n \geqs 0$ for all $m \geqs 0$. If $n=0$ and $m \geqs 0$, then we have
 \begin{align*}
  D_{m,0}(q) \
  &\stackrel{\mathclap{\eqref{E:D_def}}}{=} \ C_{-2m-1,m}(q^{-1}).
 \end{align*}
 This implies
 \begin{align*}
  h(q^{-1})^{\left\lfloor \frac{m+1}{2} \right\rfloor} &\mid D_{m,0}(q),
 \end{align*}
 and the claim follows from $h(q) \mid h(q^{-1}) = -q^{-1} h(q)$. If $n>0$ and $m \geqs r-1$, then we have
 \begin{align*}
  D_{m,n}(q) \
  &\stackrel{\mathclap{\eqref{E:D_rec}}}{=} \ \left( 1 - q^{-2(m+n)} \right) D_{m,n-1}(q) - q^{-2(2m+n+1)} D_{m+1,n-1}(q),
 \end{align*}
 and the claim follows from $h(q^{-1}) \mid \left( 1 - q^{-2(m+n)} \right)$.
 
 Finally, in order to prove Equation~\eqref{E:E_div}, we notice again that
 \begin{align*}
  \left( \left( \frac{\rmd}{\rmd q} \right)^\ell E_{m_1,m_2,n_1,n_2} \right)(1) = 
  \left( \left( \frac{\rmd}{\rmd q} \right)^\ell Q_{0,m_1,m_2,n_1,n_2} \right)(1) = 0
 \end{align*}
 for every integer $0 \leqs \ell \leqs \frac{n_1+n_2-1}{2}$, where the polynomial 
 \begin{align*}
  Q_{\ell,m_1,m_2,n_1,n_2}(q) \in \Z[q,q^{-1}]
 \end{align*}
 is defined by Equations~\eqref{E:e_def} and \eqref{E:Q_def} for all integers $\ell,n_1,n_2 \geqs 0$ and $m_1,m_2 \in \Z$. Then, the claim follows from Equation~\eqref{E:Q_div}.
\end{proof}

\end{document}